\documentclass[12pt]{article}
\usepackage{fullpage, geometry, amssymb, amsmath, amsthm, stmaryrd, setspace, graphicx}
\usepackage[all]{xy}

\newtheorem{theorem}{Theorem}[section]
\newtheorem{lemma}[theorem]{Lemma}
\newtheorem{proposition}[theorem]{Proposition}
\newtheorem{corollary}[theorem]{Corollary}
\newtheorem{definition}[theorem]{Definition}

\newtheorem{algorithm}[theorem]{Algorithm}

\newtheorem{problem}{Problem}

\newcommand{\cref}[1]{Corollary~\textup{\ref{cor:#1}}}

\newcommand{\G}{\Gamma}

\newcommand{\calP}{\mathcal P}
\newcommand{\calQ}{\mathcal Q}

\newcommand{\calK}{\mathcal K}
\newcommand{\comment}[1]{}
\newcommand{\ol}[1]{\overline{#1}}
\newcommand{\eps}{\epsilon}
\newcommand{\mix}{\diamond}
\newcommand{\comix}{\boxempty}
\newcommand{\s}{\sigma}
\newcommand{\Z}{\mathbb Z}

\DeclareMathOperator{\Norm}{Norm}

\DeclareMathOperator{\Cent}{Cent}
\DeclareMathOperator{\lcm}{lcm}
\DeclareMathOperator{\Aut}{Aut}

\begin{document}

\title{Internal and external duality in abstract polytopes}

\author{Gabe Cunningham\\
University of Massachusetts Boston\\
and \\
Mark Mixer \\
Wentworth Institute of Technology
}

\date{ \today }
\maketitle

\begin{abstract}
	We define an abstract regular polytope to be internally self-dual if its self-duality can be realized
	as one of its symmetries. This property has many interesting implications on the structure of the polytope,
	which we present here. Then, we construct many examples of internally self-dual polytopes.
	In particular, we show that there are internally self-dual regular polyhedra of each
	type $\{p, p\}$ for $p \geq 3$ and that there are both infinitely many internally self-dual and infinitely many externally self-dual polyhedra
	of type $\{p, p\}$ for $p$ even. We also show that there are internally self-dual polytopes in each rank,
	including a new family of polytopes that we construct here.
	

\comment{
	20B25: Finite automorphism groups of algebraic, geometric, or combinatorial structures.
	51M20: Polyhedra and polytopes; regular figures, division of spaces
	52B05: Combinatorial properties of polyhedra (number of faces etc.)
	52B15: Symmetry properties of polytopes
}

\end{abstract}

\section{Introduction}
Whenever a polytope is invariant under a transformation such as duality, we often describe this as an ``external'' symmetry
of the polytope. In the context of abstract regular polytopes, self-duality manifests as an automorphism of the 
symmetry group of the polytope. For example, the $n$-simplex is self-dual, and this fact is reflected by an
automorphism of its symmetry group $S_{n+1}$. Since the symmetric group $S_{n+1}$ has no nontrivial outer
automorphisms (unless $n+1 = 6$), we see that in general the self-duality of the simplex must somehow
correspond to an ordinary rank-preserving symmetry. What does this mean geometrically? What other polytopes
have this property that self-duality yields an inner automorphism?

In this paper, we define a regular self-dual abstract polytope to be \emph{internally self-dual} if
the group automorphism induced by its self-duality is an inner automorphism. Otherwise, a regular
self-dual polytope is \emph{externally self-dual}. 
Our search for internally self-dual regular polytopes started with the atlas of regular polytopes
with up to 2000 flags \cite{conder-atlas}. Using a computer algebra system, we found the following results on the number of regular polytopes (up to isomorphism and duality).

\begin{table}[h]
\begin{tabular}{ccccc}
\textbf{Rank} & \textbf{Total} & \textbf{Self-dual} & \textbf{Internally self-dual} & \textbf{Externally self-dual} \\
3 & 3571 & 242 & 54 & 188 \\
4 & 2016 & 156 &  2 & 154 \\
5 &  258 &  15 &  1 &  14 \\
\end{tabular}
\caption{Data on dualities of small regular polytopes  \label{MarstonData}}
\end{table}

The internally self-dual polyhedra we found include examples of type $\{p, p\}$ for $3 \leq p \leq 12$ and
$p = 15$. In rank 4, the only two examples are the simplex $\{3, 3, 3\}$ and a toroid of type $\{4, 3, 4\}$.
In rank 5, the only example is the simplex $\{3, 3, 3, 3\}$. 

The data suggest that there are many internally self-dual polyhedra (although many more externally self-dual). Indeed, we will show that there are internally
self-dual polyhedra of type $\{p, p\}$ for each $p \geq 3$. The data in ranks 4 and 5 seem less promising as
far as the existence of internally self-dual polytopes. We will show, however, that there are several families
of internally self-dual polytopes in every rank. Furthermore, due to the ubiquity of symmetric groups, it seems
likely that in every rank, there are many polytopes that are internally self-dual for the same reason as the simplex.

In addition to the existence results mentioned above, we explore the consequences of a self-dual polytope being
either externally or internally self-dual.

The paper is organized as follows. In section~\ref{Background} we briefly outline the necessary definitions and background on abstract polytopes.   Section~\ref{Basics} consists of basic results following from the definition of an internally self-dual polytope.    In section~\ref{Structural}, more structural results about internally self-dual polytopes are examined.  Section~\ref{Examples} includes existence results regarding both self-dual regular polyhedra and self-dual polytopes of higher ranks.  Finally, in section~\ref{SelfPetrie}, we highlight a few open questions and variants of the problems considered here.

\section{Background \label{Background}}
Most of our background, along with many more details can be found in from \cite{arp}.
An {\em abstract $n$-polytope} $\calP$ is a ranked partially ordered set of {\em faces} with four defining properties.  First, $\calP$ has a least face $F_{-1}$ of rank $-1$, and a greatest face $F_{n}$ of rank $n$; here if an $F \in \calP$ has rank$(F)=i$, then $F$ is called an \textit{$i$-face}. Faces of rank $0$, $1$, and $n-1$ are called \emph{vertices}, \emph{edges}, and \emph{facets}, respectively.  Second, each maximal totally ordered subset of $\calP$ contains $n+2$ faces.  These maximal totally ordered subsets of $\calP$ are called {\em flags}.   Third, if $F < G$ with rank$(F) = i-1$ and rank$(G)=i+1$, then there are exactly two $i$-faces $H$ with $F < H < G$.  Finally, $\calP$ is strongly connected, which is defined as follows.  For any two faces $F$ and $G$ of $\calP$ with $F < G$, we call $G/F:=\{H \mid H \in \calP, \  F \leq H \leq G\}$ a \textit{section} of $\calP$.  If $\calP$ is a partially ordered set satisfying the first two properties, then $\calP$ is \textit{connected\/} if either $n \leq 1$, or $n \geq 2$ and for any two faces $F$ and $G$ of $\calP$ (other than $F_{-1}$ and $F_{n}$) there is a sequence of faces $F = H_0, H_1,\ldots,H_{k-1},H_k = G$, not containing  $F_{-1}$ and $F_{n}$, such that $H_i$ and $H_{i-1}$ are comparable for $i=1,\ldots,k$. We say that $\calP$ is \textit{strongly connected\/} if each section of $\calP$ (including $\calP$ itself) is connected.  Due to the relationship between abstract polytopes and incidence geometries, if two faces $F$ and $G$ are comparable in the partial order, then it is said that $F$ is \emph{incident on} $G$.

Two flags of an $n$-polytope $\calP$ are said to be \textit{adjacent\/} if they differ by exactly one face. If $\Phi$ is a flag of $\calP$, the third defining property of an abstract polytope tells us that for $i=0,1,\ldots,n-1$ there is exactly one flag that differs from $\Phi$ in its $i$-face. This flag is denoted $\Phi^i$ and is {\em i-adjacent\/} to $\Phi$. We extend this notation recursively and let $\Phi^{i_1 \cdots i_k}$ denote the flag $(\Phi^{i_1 \cdots i_{k-1}})^{i_k}$. If $w$ is a finite sequence $(i_1, \ldots, i_k)$, then $\Phi^w$ denotes $\Phi^{i_1 \cdots i_k}$. 
Note that $\Phi^{ii}=\Phi$ for each $i$, and $\Phi^{ij}=\Phi^{ji}$ if $|i-j|>1$. 
An $n$-polytope $(n \geq 2)$ is called \textit{equivelar} if for each $i=1,2,...,d-1$, there is an integer $p_i$ so that every section $G/F$ defined by an $(i-2)$-face $F$ incident on an $(i+1)$-face $G$ is a $p_i$-gon. If $\calP$ is equivelar we say that it has (\textit{Schl\"{a}fli\/}) {\em type} $\{p_1,p_2,\ldots,p_{n-1}\}$.  

The automorphism group of an $n$-polytope $\calP$ (consisting of the order-preserving bijections from $\calP$ to itself) is denoted by $\Gamma( \calP )$.  For any flag $\Phi$, any finite sequence $w$, and any automorphism $\varphi$, we have $\Phi^w \varphi = (\Phi \varphi)^w$. An $n$-polytope $\calP$ is called \textit{regular} if its automorphism group $\Gamma( \calP )$ has exactly one orbit when acting on its flags of $\calP$, or equivalently, if for some \emph{base flag} $\Phi = \{F_{-1},F_0,F_1,\ldots,F_n\}$ and each $i=0,1,\ldots,n-1$ there exists a unique involutory automorphism $\rho_i \in \Gamma(\calP)$ such that $\Phi \rho_i =\Phi^i$.  For a regular $n$-polytope $\calP$, its group $\Gamma( \calP )$ is generated by the involutions $\rho_0,\ldots,\rho_{n-1}$ described above. These generators satisfy the relations
\begin{equation}
\label{coxrel}
(\rho_i \rho_j)^{p_{ij}}=\eps \quad (0\leq i,j \leq n-1),
\end{equation}
where $p_{ii}=1$ for all $i$, $2 \leq p_{ji} = p_{ij} $ if $j=i-1$, and  
\begin{equation}
\label{string}
p_{ij}=2 \textrm{ for } |i-j| \geq 2.
\end{equation}

Any group $\langle \rho_0,\ldots,\rho_{n-1} \rangle$ satisfying properties~\ref{coxrel} and \ref{string} is called a {\em string group generated by involutions}. Moreover, $\Gamma( \calP )$ has the following \textit{intersection property}:
\begin{equation}
\label{intprop}
\langle \rho_i \mid i\in I \rangle \cap \langle \rho_i \mid i\in J \rangle 
= \langle \rho_i \mid i\in  I  \cap J \rangle \,\textrm{ for }\, I,J \subseteq \{0,\ldots,n-1\}.
\end{equation}
Any string group generated by involutions that has this intersection property is called a {\em string C-group}. The group $\Gamma( \calP )$ of an abstract regular polytope $\calP$ is a string C-group.  Conversely, it is known (see \cite[Sec. 2E]{arp}) that an abstract regular $n$-polytope can be constructed uniquely from any string C-group.  This correspondence between string C-groups and automorphism groups of abstract regular polytopes allows us to talk simulatneously about the combinatorics of the flags of abstract regular polytopes as well as the structure of their groups.

Let $\Gamma=\langle\rho_0,\,\ldots,\,\rho_{n-1}\rangle$ be a string group generated by involutions acting as a permutation group on a set $\{1,\,\ldots,\,k\}$.
We can define the {\em permutation representation graph} $\mathcal{X}$ as the $r$-edge-labeled multigraph with $k$ vertices, and with a single $i$-edge $\{a,\,b\}$ whenever $a\rho_i=b$ with $a < b$.  Note that, since each of the generators is an involution, the edges in our graph are not directed.  When $\Gamma$ is a string C-group, the multigraph $\mathcal{X}$ is called a \emph{CPR graph}, as defined in \cite{DanielCPR}.

If $\calP$ and $\calQ$ are regular polytopes, then we say that $\calP$ \emph{covers} $\calQ$ if there is a well-defined
surjective homomorphism from $\G(\calP)$ to $\G(\calQ)$ that respects the canonical generators. In other words,
if $\G(\calP) = \langle \rho_0, \ldots, \rho_{n-1} \rangle$ and $\G(\calQ) = \langle \rho_0', \ldots,
\rho_{n-1}' \rangle$, then $\calP$ covers $\calQ$ if there is a homomorphism that sends each $\rho_i$
to $\rho_i'$. 

Suppose $\G = \langle \rho_0, \ldots, \rho_{n-1} \rangle$ is a string group generated by involutions, and $\Lambda$ 
is a string C-group such that $\G$ covers $\Lambda$. If the covering is one-to-one on the subgroup $\langle \rho_0, \ldots, 
\rho_{n-2} \rangle$, then the \emph{quotient criterion} says that $\G$ is itself a string C-group \cite[Thm. 2E17]{arp}.

Given regular polytopes $\calP$ and $\calQ$, the \emph{mix} of their automorphism groups, denoted $\G(\calP) \mix
\G(\calQ)$, is the subgroup of the direct product $\G(\calP) \times \G(\calQ)$ that is generated by the
elements $(\rho_i, \rho_i')$. This group is the minimal string group generated by involutions that covers
both $\G(\calP)$ and $\G(\calQ)$ (where again, we only consider homomorphisms that respect the generators). 
Using the procedure in \cite[Sec. 2E]{arp}, we can build a poset from this, which
we call the mix of $\calP$ and $\calQ$, denoted $\calP \mix \calQ$. This definition naturally extends
to any family of polytopes (even an infinite family); see \cite[Section 5]{mixing-and-monodromy} for more details.

The \emph{comix} of $\G(\calP)$ with $\G(\calQ)$, denoted $\G(\calP) \comix \G(\calQ)$, is the largest
string group generated by involutions that is covered by both $\G(\calP)$ and $\G(\calQ)$ \cite{var-gps}. A presentation
for $\G(\calP) \comix \G(\calQ)$ can be obtained from that of $\G(\calP)$ by adding all of the relations of
$\G(\calQ)$, rewriting the relations to use the generators of $\G(\calP)$ instead.

The dual of a poset $\calP$ is the poset $\calP^*$ with the same underlying set as $\calP$, but with the partial order
reversed. Clearly, the dual of a polytope is itself a polytope. If $\calP^* \cong \calP$, then $\calP$ is said
to be \emph{self-dual}. A \emph{duality} $d: \calP \to \calP$ is an order-reversing bijection. 
If $w = (i_1, \ldots, i_k)$, then we define $w^*$ to be the sequence $(n-i_1-1, \ldots, n-i_k-1)$.
For any flag
$\Phi$, finite sequence $w$, and duality $d$, we have $\Phi^w d = (\Phi d)^{w*}$.

If $\calP$ is a self-dual regular $n$-polytope, then there is a group automorphism of $\G(\calP)$ that
sends each $\rho_i$ to $\rho_{n-i-1}$. 
If $\varphi \in \G(\calP)$, then we will denote by $\varphi^*$
the image of $\varphi$ under this automorphism. In particular, if $\varphi = \rho_{i_1} \cdots \rho_{i_k}$, 
then we define $\varphi^{*}$ to be $\rho_{n-i_1-1} \cdots \rho_{n-i_k-1}$.
Thus, if $\Phi$ is the base flag of $\calP$ and $\Phi \varphi = \Phi^w$, then $\Phi \varphi^*
= \Phi^{w*}$.

\section{Internal self-duality \label{Basics}}
\subsection{Basic notions}

As just noted, if $\calP$ is a self-dual regular $n$-polytope, then there is a group automorphism of
$\G(\calP)$ that sends each $\rho_i$ to $\rho_{n-i-1}$. When this automorphism is
inner, there is a polytope symmetry $\alpha \in \G(\calP)$ such that
$\alpha \rho_i = \rho_{n-i-1} \alpha$ for $0 \leq i \leq n-1$. We use this idea to motivate the following definitions.

\begin{definition}
\label{def-int-sd}
Suppose $\calP$ is a regular self-dual polytope. An automorphism $\alpha \in \G(\calP)$ is
a \emph{dualizing automorphism} (or simply \emph{dualizing}) if $\alpha \rho_i = \rho_{n-i-1} \alpha$
for $0 \leq i \leq n-1$ (equivalently, if $\alpha \varphi = \varphi^* \alpha$ for every $\varphi \in \G(\calP)$).
A regular self-dual polytope is \emph{internally self-dual} if it has a dualizing automorphism.
Otherwise, it is \emph{externally self-dual}.
\end{definition}

Depending on the automorphism group of a polytope, we can sometimes determine internal self-duality
without knowing anything deep about the polytope's structure.

\begin{proposition}
\label{symmetric}
Suppose $\calP$ is a regular self-dual polytope, and that $\G(\calP)$ is a symmetric group.
Then $\calP$ is internally self-dual unless $\G(\calP) \cong S_6$ and $\calP$ has type
$\{6, 6\}$ or $\{4, 4, 4\}$.
\end{proposition}

\begin{proof}
For $k \neq 6$, the symmetric group $S_k$ has no nontrivial outer automorphisms, and so a self-dual
polytope with this automorphism group must be internally self-dual. Up to isomorphism, there are
11 regular polytopes with automorphism group $S_6$; eight of them are not self-dual, one of them
(the 5-simplex) is internally self-dual, and the remaining two (denoted $\mbox{\{6,6\}*720a}$ and $\mbox{\{4, 4, 4\}*720}$)
are externally self dual \cite{atlas}.
\end{proof}

It would perhaps be possible to find other abstract results of this type, but they could never tell the whole
story. After all, it is possible for the automorphism that sends each $\rho_i$ to $\rho_{n-i-1}$ to be inner, even
if $\G(\calP)$ has nontrivial outer automorphisms. So let us shift our focus
away from the abstract groups.

What does it mean geometrically to say that $\calP$ is internally self-dual? Let $\Phi$ be
the base flag of $\calP$, and let $\Psi = \Phi \alpha$ for some dualizing automorphism $\alpha$.
Let~$\varphi \in \G(\calP)$ and suppose that $\Phi \varphi = \Phi^w$. Then
\begin{align*}
\Psi \varphi &= 
\Phi \alpha \varphi \\
&= \Phi \varphi^* \alpha \\
&= \Phi^{w^*} \alpha \\
&= (\Phi \alpha)^{w^*} \\
&= \Psi^{w^*}.
\end{align*}
In other words, every automorphism acts on $\Psi$ dually to how it acts on $\Phi$.

\begin{definition}
\label{def-dual-flag}
We say that flags $\Phi$ and $\Psi$ are \emph{dual} (to each other) if every automorphism acts
dually on $\Phi$ and $\Psi$. That is, if $\Phi \varphi = \Phi^{w}$, then
$\Psi \varphi = \Psi^{w^*}$.
\end{definition}

Note that if $\calP$ is regular and $\Psi$ is dual to $\Phi$, then in particular $\Psi \rho_i = \Psi^{n-i-1}$.

\begin{proposition}
\label{int-sd-iff-dual-flag}
A regular polytope $\calP$ is internally self-dual if and only if its base flag $\Phi$ has a dual flag $\Psi$.
\end{proposition}

\begin{proof}
The discussion preceding Definition~\ref{def-dual-flag} shows that if $\calP$ is internally self-dual, then
$\Phi$ has a dual flag. Conversely, suppose that $\Phi$ has a dual flag $\Psi$.
Since $\calP$ is regular, there is an automorphism $\alpha \in \G(\calP)$ such that $\Psi = \Phi \alpha$.
Then for $0 \leq i \leq n-1$,
\begin{align*}
\Phi \alpha \rho_i \alpha^{-1} &= \Psi \rho_i \alpha^{-1} \\
&= \Psi^{n-i-1} \alpha^{-1} \\
&= (\Psi \alpha^{-1})^{n-i-1} \\
&= \Phi^{n-i-1}.
\end{align*}
So $\alpha \rho_i \alpha^{-1}$ acts on $\Phi$ the same way that $\rho_{n-i-1}$ does, and since the
automorphism group acts regularly on the flags, it follows that $\alpha \rho_i \alpha^{-1} = \rho_{n-i-1}$
for each $i$. Thus $\alpha$ is a dualizing automorphism, and $\calP$ is internally self-dual.
\end{proof}

Proposition~\ref{int-sd-iff-dual-flag} provides an intuitive way to determine whether a regular polytope
is internally self-dual: try to find a flag that is dual to the base flag. 
Let us consider some simple examples. Let $\calP = \{p\}$, with $3 \leq p < \infty$. Fix
a vertex and an edge in the base flag $\Phi$. In order for $\Psi$ to be dual to $\Phi$, we need
for $\Psi \rho_0 = \Psi^1$, which in particular means that $\rho_0$ fixes the vertex of $\Psi$.
Now, whenever $p$ is even or infinite, the reflection $\rho_0$ does not fix any vertices, and so
$\calP$ must be externally self dual (See Figure~\ref{figure-pgons}). When $p$ is odd, then there is a unique vertex $v$
fixed by $\rho_0$. Furthermore, there is an edge incident to $v$ that is fixed by $\rho_1$. The flag $\Psi$
consisting of this vertex and edge is dual to $\Phi$, and so in this case, $\calP$ is internally self-dual.
Indeed, when $p$ is odd, then the automorphism $(\rho_0 \rho_1)^{(p-1)/2} \rho_0$ is dualizing.
The following result is then clear.

\begin{proposition}
\label{p-gons}
The $p$-gon $\calP = \{p\}$ is internally self-dual if and only if $p$ is odd.
\end{proposition}

\begin{figure}[h]
$$\includegraphics[scale=.25]{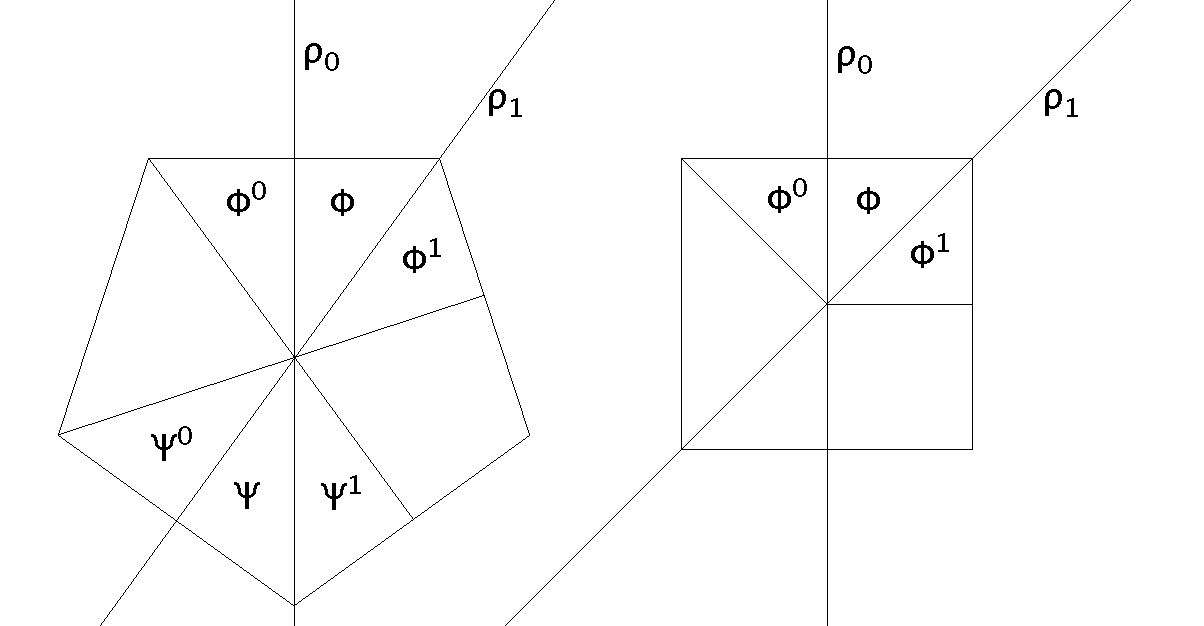} $$
\caption{Adjacent flags and dual flags in a pentagon and a square \label{figure-pgons}
}
\end{figure}

Next consider $\calP = \{3, 3\}$, the simplex. Since $\G(\calP) \cong S_4$, which has no nontrivial
outer automorphisms, it follows that $\calP$ is internally self-dual. Nevertheless, let us see
what the dual to the base flag looks like.  Consider the labeled simplex in Figure~\ref{LabeledSimplex}, with vertices $\{1,2,3,4\}$, edges $\{a,b,c,d,e,f\}$, and facets $\{L,F, R, D\}$.  Let us pick the triple $\Phi =(1,a,L)$ as the base flag.

\begin{figure}[h]
$$\includegraphics[scale=.5]{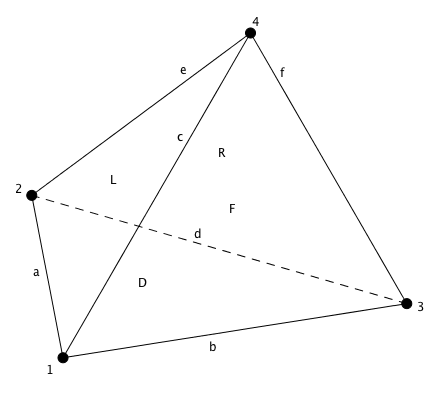}$$
\caption{A simplex with labeled faces \label{LabeledSimplex}}
\end{figure}

To find the dual of the base flag of this simplex consider the action of each of the distinguished generators of the automorphism group on the vertices, edges, and facets.  We summarize the actions in Table~\ref{table-simplex-actions}.

\begin{table}[h]
\begin{center}
\begin{tabular}{|c|c|c|c|}
\hline
 & Vertices & Edges & Facets \\ \hline
 $\rho_0$ & $(1,2)(3)(4)$ & $(a)(b,d)(c,e)(f)$ & $(L)(F,R)(D)$ \\ \hline
 $\rho_1$ & $(1)(2,4)(3)$ & $(a,c)(b)(d,f)(e)$ & $(L)(F,D)(R)$ \\ \hline 
 $\rho_2$ & $(1)(2)(3,4)$ & $(a)(b,c)(d,e)(f)$ & $(L,D)(F)(R)$ \\ \hline
 \end{tabular}
  \end{center}
  \caption{Symmetries of the labeled simplex \label{table-simplex-actions}}
 \end{table}

Since $\rho_1$ and $\rho_2$ both fix the base vertex, we need for their duals, $\rho_1$ and $\rho_0$, to fix
the vertex of the dual flag. The only possibility is vertex 3. Next, since $\rho_0$ and $\rho_2$ fix the base
edge, we need for their duals to fix the edge of the dual flag. The only edges fixed by both $\rho_0$ and
$\rho_2$ are edges $a$ and $f$, and the only one of those incident on vertex 3 is $f$. Finally, since
$\rho_0$ and $\rho_1$ fix the base facet, we need for $\rho_2$ and $\rho_1$ to fix the facet of the dual flag,
and so the only possibility is $R$. So the flag that is dual to $(1, a, L)$ is $(3, f, R)$ (shown in Figure~\ref{LabeledSimplexDual}).
 
\comment{
	First let us determine the vertex of the flag $\Psi$ which is dual to our base flag $\Phi$.   Since $\rho_0$ has our base vertex in its support, $\rho_2$ must have the vertex of the dual flag in its support.  Thus the vertex of the dual flag must be either vertex 3 or 4.  Furthermore, $\rho_1$ fixes our base vertex, so $\rho_1$ must also fix the vertex of the dual flag.  We can conclude that the vertex of $\Psi$ is vertex 3.

	Next, we can determine the face of $\Psi$ in a similar manner.  Since $\rho_0$ fixes the base face $L$, $\rho_2$ must fix the face of $\Psi$.  Thus the face of $\Psi$ is either $F$ or $R$.  Similarly, $\rho_1$ fixes the base face $L$, and so $\rho_1$ must fix the face of $\Psi$.  We can conclude that the  face of $\Psi$ is $R$.

	There are two edges ($d$ and $f$) incident to both face $R$ and vertex 3.  All that is left is to determine which of them must be the edge of $\Psi$ and we have determined the flag dual to our base flag.  Since $\rho_0$ fixes our base edge, $\rho_2$ must fix the edge of $\Psi$.  Finally, we conclude that the edge of $\Psi$ is $f$, and in summary $\Psi = (3, f, R)$ as shown below.
}

\begin{figure}[h]
$$\includegraphics[scale=.5]{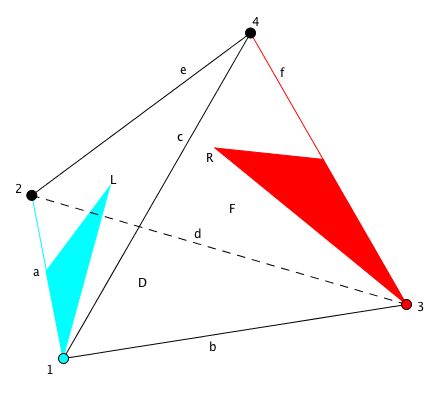}$$
\caption{A base and dual flag of a simplex \label{LabeledSimplexDual}}
\end{figure}

	The process just described is a good illustration of the general process of finding
	a dual flag. Let us now describe that process. 
	Suppose $\calP$ is a regular, self-dual polytope, with base flag $\Phi$ and
	$\G(\calP) = \langle \rho_0, \ldots, \rho_{n-1} \rangle$. 
	A flag $\Psi$ will be dual to $\Phi$ if and only if, for $0 \leq i \leq n-1$, the 
	$i$-face of $\Psi$ is fixed by $\langle \rho_j \mid j \neq n-i-1 \rangle$. 
	So to find $\Psi$, we start by looking for a vertex $F_0$ that is fixed by $\langle \rho_0, \ldots, \rho_{n-2} \rangle$.
	If no such vertex exists, then $\Phi$ does not have a dual flag. Otherwise, once we have found
	$F_0$, we now need an edge $F_1$ that is incident to $F_0$ and 
	fixed by $\langle \rho_0, \ldots, \rho_{n-3}, \rho_{n-1} \rangle$.
	Since $\rho_{n-1}$ does not fix $F_0$, the only way that it will fix $F_1$ is if it interchanges
	the endpoints. Thus, $F_1$ must be incident on the vertices $F_0$ and $F_0 \rho_{n-1}$. 
	Similar reasoning then shows that $F_2$ must be incident on all of the edges in
	$F_1 \langle \rho_{n-2}, \rho_{n-1} \rangle$. Continuing in this way, we arrive at the following
	algorithm:
	\comment{
		We note that,
		if $\varphi \in \langle \rho_0, \ldots, \rho_{n-3}, \rho_{n-1} \rangle$, then $\varphi$ either
		fixes both $F_0$ and $F_0 \rho_{n-1}$, or it interchanges them. If $F_1$ is the only edge
		incident to $F_0$ and $F_0 \rho_{n-1}$, then it follows that $\varphi$ fixes it.
		Otherwise, in principle, $\varphi$ might exchange $F_1$ with another edge incident to
		$F_0$ and $F_0 \rho_{n-1}$.
	}
	\begin{algorithm}
	\label{dual-flag-proc}
	Suppose $\calP$ is a regular $n$-polytope with base flag $\Phi$ and $\G(\calP) = \langle \rho_0, \ldots, \rho_{n-1}
	\rangle$. We pick a flag $\Psi = (F_0, F_1, \ldots, F_{n-1})$ as follows.
	\begin{enumerate}
	\item[$1$.] Set $F_0$ to be a vertex of $\calP$ that is fixed by $\langle \rho_0, \ldots, \rho_{n-2} \rangle$.
	\item[$2$.] Once $F_i$ is chosen, pick $F_{i+1}$ to be an $(i+1)$-face that is incident to every
	$i$-face in $F_i \langle \rho_{n-i-1}, \ldots, \rho_{n-1} \rangle$.
	\item[$3$.] Repeat Step 2 until $F_{n-1}$ is chosen.
	\end{enumerate}
	\end{algorithm}
	
	For simplicity, let us assume for now that $\calP$ is \emph{vertex-describable}, meaning that each face
	of $\calP$ is uniquely determined by its vertex-set. 
	
	\begin{proposition}
	\label{find-dual-flag}
	Suppose $\calP$ is a regular, self-dual, vertex-describable polytope.
	Then $\calP$ is internally self-dual if and only if the procedure described in Algorithm~\ref{dual-flag-proc}
	finishes successfully.
	\end{proposition}

	\begin{proof}
	If $\calP$ is internally self-dual, then its base flag $\Phi$ has a dual flag $\Psi$. Let
	$\Psi = (F_0, \ldots, F_{n-1})$. Then for each $i$, the face $F_i$ is fixed by $\langle \rho_j \mid j \neq n-i-1 \rangle$.
	In particular, since $F_i < F_{i+1}$, it follows that each $\varphi \in \langle \rho_{n-i-1}, \ldots, \rho_{n-1} \rangle$, we have $F_i \varphi < F_{i+1} \varphi = F_{i+1}$. Thus, once we have chosen $F_i$,
	the face $F_{i+1}$ satisfies the desired properties, and we can continue the process until we
	have produced $\Psi$.
	
	Conversely, suppose that the algorithm produced a sequence of faces $(F_0, \ldots, F_{n-1})$. 
	To show that $\calP$ is internally self-dual, it suffices to show that $\rho_j$ fixes every face except for $F_{n-j-1}$. 
	By construction, the vertex-set of each face $F_i$ is $F_0 \langle \rho_{n-i}, \ldots, \rho_{n-1} \rangle$.
	For $j < n-i-1$, the automorphism $\rho_j$ commutes with $\langle \rho_{n-i}, \ldots, \rho_{n-1} \rangle$ and fixes
	$F_0$, and so
	\[ F_0 \langle \rho_{n-i}, \ldots, \rho_{n-1} \rangle \rho_j = F_0 \rho_j \langle \rho_{n-i}, \ldots, \rho_{n-1} \rangle = F_0 \langle \rho_{n-i}, \ldots, \rho_{n-1} \rangle. \]
	For $j > n-i-1$, we have
	\[ F_0 \langle \rho_{n-i}, \ldots, \rho_{n-1} \rangle \rho_j.  = F_0 \langle \rho_{n-i}, \ldots, \rho_{n-1} \rangle. \]
	So for $j \neq n-i-1$, the automorphism $\rho_j$ fixes the vertex set of $F_i$. Since $\calP$ is vertex-describable,
	it follows that $\rho_j$ fixes $F_i$ itself. 
	Thus $\rho_j$ fixes every face except (possibly) for $F_{n-j-1}$.
	But it is clear that $\rho_j$ cannot also fix $F_{n-j-1}$, because the only automorphism
	that fixes any flag is the identity. This shows that $\Psi = (F_0, \ldots, F_{n-1})$ is
	dual to $\Phi$.
	\end{proof}
	
	\begin{corollary}
	\label{fixed-vertices}
	Let $\calP$ be a self-dual regular polytope with $\G(\calP) = \langle \rho_0, \ldots, \rho_{n-1} \rangle$.
	If no vertices are fixed by the subgroup $\langle \rho_0, \ldots, \rho_{n-2} \rangle$,
	then $\calP$ is externally self-dual.
	\end{corollary}

\subsection{Properties of dualizing automorphisms}

Now let us return to the algebraic point of view to determine some properties of dualizing automorphisms.

\begin{proposition}
\label{dualizing-self-dual}
If $\alpha \in \G(\calP)$ is dualizing, then $\alpha = \alpha^{*}$.
\end{proposition}

\begin{proof}
If $\alpha$ is dualizing, then for all $\varphi \in \G(\calP)$, we have that $\alpha \varphi = \varphi^{*} \alpha$. Taking $\varphi = \alpha$ yields the desired result.
\end{proof}
	
\begin{proposition}
If $\alpha \in \G(\calP)$ is dualizing, then $\alpha$ is not in $\langle \rho_0, \ldots, \rho_{n-2} \rangle$
or in $\langle \rho_1, \ldots, \rho_{n-1} \rangle$.
\end{proposition}
	
\begin{proof}
Since $\alpha$ is dualizing, $\rho_{n-1} =  \alpha^{-1} \rho_{0} \alpha$. If $\alpha \in
\langle \rho_0, \ldots, \rho_{n-2} \rangle$, then this gives us that $\rho_{n-1}$ is in $\langle \rho_0, \ldots,
\rho_{n-2} \rangle$, which violates the intersection condition (Equation~\ref{intprop}). Similarly, if $\alpha \in \langle \rho_1,
\ldots, \rho_{n-1} \rangle$, then the equation $\rho_0 = \alpha^{-1} \rho_{n-1} \alpha$ shows that
$\rho_0$ is in $\langle \rho_1, \ldots, \rho_{n-1} \rangle$, which again violates the intersection
condition.
\end{proof}

The following properties are straightforward to verify.

\begin{proposition}
\label{dualizing-props}
Let $\calP$ be a self-dual regular polytope.
\begin{enumerate} 
\item If $\alpha$ and $\beta$ are dualizing automorphisms, then $\alpha \beta$ is central, and $\alpha \beta = \beta \alpha$.
\item If $\varphi$ is central and $\alpha$ is dualizing, then $\varphi \alpha$ is dualizing.
\item If $\alpha$ is dualizing, then any even power of $\alpha$ is central, and any odd power of $\alpha$ is dualizing. In particular, $\alpha$ has even order.
\item If $\alpha$ is dualizing, then $\alpha^{-1}$ is dualizing.
\end{enumerate}
\end{proposition}

\comment{
	\begin{proof}
	\begin{enumerate}
	\item For all $i$, we have
			\[ \rho_i \alpha \beta = \alpha \rho_{n-i-1} \beta = \alpha \beta \rho_i, \]
	so $\alpha \beta$ is central. Then it follows that $(\alpha \beta) \alpha^{-1} = \alpha^{-1} (\alpha \beta) = \beta$, 
	and so $\alpha \beta = \beta \alpha$.
	\item For all $i$, we have
			\[ \rho_i \varphi \alpha = \varphi \rho_i \alpha = \varphi \alpha \rho_{n-i-1}. \]
	\item Apply the previous two results, and note that $1$ cannot be dualizing, since that would imply that
	$\rho_0 = \rho_{n-1}$.
	\item Follows from the previous result.
	\end{enumerate}
	\end{proof}
}

\begin{proposition}
Let $\calP$ be an internally self-dual regular polytope such that $\G(\calP)$ has a finite center. 
Then the number of dualizing automorphisms of $\G(\calP)$ is equal to the order of the center of $\G(\calP)$.
\end{proposition}

\begin{proof}
Proposition \ref{dualizing-props} implies that the central and dualizing automorphisms of $\G(\calP)$
together form a group in which the center of $\G(\calP)$ has index 2. The result follows immediately.
\end{proof}

\subsection{Internal self-duality of nonregular polytopes}

Generalizing internal self-duality to nonregular polytopes is not entirely straightforward.
When $\calP$ is a self-dual regular polytope, then $\G(\calP)$ always has an automorphism
(inner or outer) that reflects this self-duality. This is not the case for general polytopes.

One promising way to generalize internal self-duality is using the notion of dual flags.
Indeed, Definition~\ref{def-dual-flag} does not require the polytope to be regular, and
makes sense even for nonregular polytopes. Let us determine some of the simple
consequences of this definition.

\begin{proposition}
Suppose that $\Psi$ is dual to $\Phi$. Then $\Psi^{n-i-1}$ is dual to $\Phi^i$.
\end{proposition}
	
\begin{proof}
Let $\varphi \in \G(\calP)$. We need to show that $\varphi$ acts dually on $\Phi^i$ and
$\Psi^{n-i-1}$. Suppose that $\Phi \varphi = \Phi^{w}$, from which it follows that
$\Psi \varphi = \Psi^{w^{*}}$. Then:
\[ (\Phi^i) \varphi = (\Phi \varphi)^i = \Phi^{w i} = (\Phi^i)^{i w i}, \]
whereas
\[ (\Psi^{n-i-1} \varphi) = (\Psi \varphi)^{n-i-1} = \Psi^{w^{*}(n-i-1)} = 
(\Psi^{n-i-1})^{(n-i-1)w^{*}(n-i-1)}, \]
and the claim follows.
\end{proof}

Since polytopes are flag-connected, the following is an immediate consequence.

\begin{corollary}
If $\calP$ has one flag that has a dual, then every flag has a dual.
\end{corollary}

Thus we see that the existence of dual flags is in fact a global property, not a local one. Here are some further consequences of the definition of dual flags.

\begin{proposition}
Let $\calP$ be a polytope, and let $\Phi$ and $\Psi$ be flags of $\calP$.
\begin{enumerate}
\item If $\Phi$ and $\Psi$ are dual, then $\calP$ is self-dual, and there is a duality $d: \calP \to \calP$ that
takes $\Phi$ to $\Psi$.
\item $\Phi$ is dual to $\Psi$ if and only if $\Psi = \Phi d$ for some duality $d$ that commutes
	with every automorphism of $\calP$.
\item $\Phi$ is dual to $\Psi$ if and only if $\Psi = \Phi d$ for some duality $d$ that 
	is central in the extended automorphism group of $\calP$.
\end{enumerate}
\end{proposition}

\begin{proof}
\begin{enumerate}
\item We attempt to define the duality $d$ by $\Phi d = \Psi$ and then extend it by $\Phi^{w} d = 
\Psi^{w^{*}}$. To check that this is well-defined, suppose that $\Phi^w = \Phi$; we then need to show
that $\Psi^{w^*} = \Psi$. Indeed, if $\Phi^w = \Phi$, then taking $\varphi$ to be the identity we have
that $\Phi \varphi = \Phi^w$, whence $\Psi = \Psi \varphi = \Psi^{w*}$, with the last equality following
since $\Phi$ and $\Psi$ are dual.

\item 	Suppose $\Phi$ and $\Psi$ are dual. Then the first part shows that
	$\Psi = \Phi d$ for some duality $d$. Now, let $\varphi$ be an automorphism, and suppose
	$\Phi \varphi = \Phi^w$. We get:
	\[ \Phi \varphi d = (\Phi^w) d = (\Phi d)^{w^{*}} = \Psi^{w^{*}} = \Psi \varphi = \Phi d \varphi. \]
	Thus the automorphisms $d^{-1} \varphi d$ and $\varphi$ both act the same way on $\Phi$, and so they
	are equal. So $\varphi d = d \varphi$, and $d$ commutes with every automorphism.
	Conversely, if $d$ commutes with every automorphism, then
	\[ \Psi \varphi = \Phi d \varphi = \Phi \varphi d = (\Phi^w) d = (\Phi d)^{w^{*}} = \Psi^{w^{*}}, \]
	and $\Psi$ is dual to $\Phi$.
\item In light of the previous part, all that remains to show is that, if a duality commutes with every automorphism
of $\calP$, then it also commutes with every duality. If $d_2$ is a duality, then $d_2 d^{-1}$ is an automorphism, and
thus:
	\[ d d_2 d^{-1} = d (d_2 d^{-1}) = (d_2 d^{-1}) d = d_2, \]
	and so $d$ commutes with $d_2$
\end{enumerate}
\end{proof}

We see that dual flags have several nice properties, even in the nonregular case. 
We could define a polytope to be internally self-dual if every flag has a dual (equivalently,
if any single flag has a dual). It is not entirely clear if this is the ``right'' definition.
In any case, we do not pursue the nonregular case any further here.

\section{Properties of internally self-dual polytopes \label{Structural}}

\subsection{Basic structural results}

	Internally self-dual polytopes have a number of structural restrictions. Many of them are consequences
	of the following simple property.
			
	\begin{proposition}
	\label{int-self-dual-covers}
	If $\calP$ is an internally self-dual regular polytope, then any regular polytope covered
	by $\calP$ is also internally self-dual.
	\end{proposition}
			
	\begin{proof}
	Since $\calP$ is internally self-dual, there is an automorphism $\alpha \in \G(\calP)$ such that
	$\alpha \rho_i = \rho_{n-i-1} \alpha$ for $0 \leq i \leq n-1$. If $\calP$ covers $\calQ$, then the image of
	$\alpha$ in $\G(\calQ)$ has the same property, and so $\calQ$ is internally self-dual.
	\end{proof}
		
	Proposition~\ref{int-self-dual-covers} makes it difficult to construct internally self-dual
	polytopes with mixing, since the mix of any two polytopes must cover them both.
	The next proposition characterizes which pairs of polytopes can be mixed
	to create an internally self-dual polytope.
		
	\begin{proposition}
	\label{int-self-dual-mix}
	Let $\calP$ and $\calQ$ be regular $n$-polytopes such that $\calP \mix \calQ$ is polytopal. Then $\calP \mix \calQ$ is internally self-dual if and only if 
	\begin{enumerate}
	\item $\calP$ and $\calQ$ are internally self-dual, and
	\item There are dualizing automorphisms $\alpha \in \G(\calP)$ and $\beta \in \G(\calQ)$ such that the images of $\alpha$ and $\beta$ in $\G(\calP) \comix \G(\calQ)$ coincide.
	\end{enumerate}
	\end{proposition}
			
	\begin{proof}
	Let $\G(\calP) = \langle \rho_0, \ldots, \rho_{n-1} \rangle$ and $\G(\calQ) = \langle \rho_0', \ldots, \rho_{n-1}' 
	\rangle$, and let $\G(\calP) \mix \G(\calQ) = \langle \s_0, \ldots, \s_{n-1} \rangle$, where $\s_i = (\rho_i, \rho_i')$. 
	Let $\varphi = (\alpha, \beta) \in \G(\calP) \mix \G(\calQ)$.
	Then $\varphi$ is dualizing if and only if $\alpha$ and $\beta$ are both dualizing, since
	$\varphi \s_i = \s_{n-i-1} \varphi$ if and only if $(\alpha \rho_i, \beta \rho_i') = (\rho_{n-i-1} \alpha,
	\rho_{n-i-1}' \beta)$. Therefore, $\calP \mix \calQ$ is internally self-dual if and only if there are dualizing 
	automorphisms $\alpha \in \G(\calP)$ and $\beta \in \G(\calQ)$ such that $(\alpha, \beta) \in \G(\calP) \mix \G(\calQ)$. 
	By \cite[Prop. 3.7]{var-gps}, this occurs if and only if the images of $\alpha$ and $\beta$ in $\G(\calP) \comix
	\G(\calQ)$ coincide.
	\end{proof}

	\begin{corollary}
	If $\calP$ and $\calQ$ are internally self-dual regular $n$-polytopes with ``the same'' dualizing
	element (i.e., the image of the same word in the free group on $\rho_0, \ldots, \rho_{n-1}$),
	then $\calP \mix \calQ$ is internally self-dual.
	\end{corollary}
	
	\comment{
		\begin{corollary}
		Let $\calP$ and $\calQ$ be internally self-dual regular $n$-polytopes such that
		$\calP \mix \calQ$ is polytopal. If $\calP$ and $\calQ$ have trivial comix, then
		$\calP \mix \calQ$ is internally self-dual.
		\end{corollary}
				
		\begin{proof}
		Apply Proposition~\ref{int-self-dual-mix}, noting that if the comix is trivial, then the condition
		(that the images of the dualizing automorphisms coincide) is trivially satisfied.
		\end{proof}
	}

	\comment{
		Note that we did not have to \emph{assume} $\calQ$ to be self-dual; that is simply a consequence of the result. 

		Actually, a stronger result holds: every quotient of $\G(\calP)$ must still have a dualizing automorphism.
	}

	If we have a presentation for $\G(\calP)$, it is often simple to show that $\calP$ is not internally
	self-dual. Indeed, because of Proposition~\ref{int-self-dual-covers}, all we need to do is find some non-self-dual
	quotient of $\G(\calP)$. For example, if $\calP$ is internally self-dual, then adding a relation that forces
	$\rho_0 = \eps$ must also force $\rho_{n-1} = \eps$. Here are some similar results that are easily applied.
	
	\begin{proposition}
	If $\calP$ is an internally self-dual regular $n$-polytope, then in the abelianization of
	$\G(\calP)$, the image of each $\rho_i$ is equal to the image of $\rho_{n-i-1}$.
	\end{proposition}
	
	\begin{proof}
	By Proposition~\ref{int-self-dual-covers}, the abelianization of $\G(\calP)$ must have a dualizing automorphism.
	Since such an automorphism also must commute with the image of every $\rho_i$, it follows that the images of
	$\rho_i$ and $\rho_{n-i-1}$ must coincide.
	\end{proof}
	
	Suppose $\calP$ is a regular $n$-polytope, with $m$-faces isomorphic to $\calK$. 
	We say that a regular $n$-polytope has the \emph{Flat Amalgamation Property (FAP) with respect to its $m$-faces} if
	adding the relations $\rho_i = \eps$ to $\G(\calP)$ for $i \geq m$ yields a presentation for $\G(\calK)$
	(rather than a proper quotient).
	
	\begin{proposition}
	\label{no-fap}
	If $\calP$ has the FAP with respect to its $m$-faces for some $m$ with $1 \leq m \leq n-1$, then
	$\calP$ is not internally self-dual.
	\end{proposition}
	
	\begin{proof}
	If $\calP$ is internally self-dual, then adding the relation $\rho_{n-1} = \eps$ to $\G(\calP)$ must
	force $\rho_0 = \eps$, and this precludes $\calP$ from having the FAP with respect to its $m$-faces
	for any $m \geq 1$.
	\end{proof}
	
	\begin{corollary} \label{mix-with-edge}
	If $\calP$ is internally self-dual and $e$ is the unique $1$-polytope, then $\G(\calP \mix e ) \cong \G(\calP) \times C_2$.
	\end{corollary}
	\begin{proof}  The group $\G(\calP \mix e )$ is either isomorphic to $\G(\calP)$ or $\G(\calP) \times C_2$.
	From Proposition~\ref{no-fap} and \cite[Thm. 7A11]{arp} we know that $\G(\calP \mix e ) \not \cong \G(\calP)$.	\end{proof}

\begin{corollary}\label{IntToExt}
If $\calP$ is internally self-dual and $e$ is the unique $1$-polytope, then $(\calP \mix e)^{*} \mix e$ is externally 
self-dual, with automorphism group $\G(\calP) \times C_2^2$.
\end{corollary}
	\begin{proof}
	In general, $(\calP \mix \calQ)^* \cong \calP^* \mix \calQ^*$, and so $(\calP \mix e)^{*} \mix e$ is self-dual. 
	Since $P$ is self-dual and $\G(\calP \mix e ) \cong \G(\calP) \times C_2$, the group $\G(\calP \mix e^* )$ is isomorphic to $ \G(\calP) \times C_2$ as well.  Thus $\G((\calP \mix e)^{*} \mix e) \cong \G(\calP) \times C^2_2$.
		 Furthermore, if you mix $(\calP \mix e)^{*} \mix e$ with an edge again, then the automorphism group does not change.  Thus, by Corollary~\ref{mix-with-edge}, $(\calP \mix e)^{*} \mix e$ must be externally self-dual.

	\comment{
		If $\calP$ is internally self-dual, then any central automorphisms must be self-dual, since dualizing
		automorphisms are supposed to dualize the central automorphisms. Thus, if we extend $\calP$ by two central
		involutions -- one at $\rho_0$ and one at $\rho_{n-1}$ -- and if we end up with a larger group
		as a result, then we have created non-self-dual central automorphisms, and so we have an externally self-dual polytope.
	}
\end{proof}

\begin{corollary}\label{even-type}
If $\calP$ is an internally self-dual polyhedron of type $\{p, p\}$, then $(\calP \mix e)^* \mix e$ is an externally
self-dual polyhedron of type $\{q, q\}$, where $q = \lcm(p, 2)$.
\end{corollary}

\begin{proof}
Let $\G(\calP) = \langle \rho_0, \rho_1, \rho_2 \rangle$ and let $\G(e) = \langle \lambda_0 \rangle$.
For $0 \leq i \leq 2$, let $\s_i = (\rho_i, \lambda_i) \in \G(\calP) \times \G(e)$, where we take $\lambda_1 =
\lambda_2 = \eps$. Then $(\s_0 \s_1)^p = ((\rho_0 \rho_1)^p, \lambda_0^p) = (\eps, \lambda_0^p)$. If $p$
is even, then this gives us $(\eps, \eps)$, and so $\s_0 \s_1$ has order $p$. Otherwise $\s_0 \s_1$ has order
$2p$. So $\calP \mix e$ is of type $\{q, p\}$, and by taking the dual and mixing with $e$ again, we
get a polyhedron of type $\{q, q\}$.
\end{proof}
	
In some sense, Corollary~\ref{IntToExt} says that externally self-dual polytopes are at least as common as internally self-dual
polytopes.	

\subsection{Internal self-duality of universal polytopes}

	A natural place to start looking for internally self-dual polytopes is the universal
	polytopes $\{p_1, \ldots, p_{n-1}\}$ whose automorphism groups are string Coxeter groups.
	Let us start with those polytopes with a 2 in their Schl\"afli symbol.
	Recall that a polytope is \emph{flat} if every vertex is incident on every facet.

	\begin{proposition}
	\label{no-flat}
	There are no flat, regular, internally self-dual polytopes.
	\end{proposition}

	\begin{proof}
	Suppose $\calP$ is a flat regular polytope, with $\G(\calP) = \langle \rho_0, \ldots, \rho_{n-1} \rangle$.
	The stabilizer of the base facet is $\G_{n-1} = \langle \rho_0, \ldots, \rho_{n-2} \rangle$, which
	acts transitively on the vertices of the base facet. Since $\calP$ is flat, it follows that $\G_{n-1}$
	acts transitively on all the vertices of $\calP$. In particular, $\G_{n-1}$ does not fix any vertices,
	and thus Corollary~\ref{fixed-vertices} implies that $\calP$ is not internally self-dual.
	\end{proof}

	\begin{corollary}
	\label{no-degenerate}
	If $\calP$ is a regular internally self-dual polytope of type $\{p_1, \ldots, p_{n-1}\}$, then
	each $p_i$ is at least $3$.
	\end{corollary}

	\begin{proof}
	Proposition 2B16 in \cite{arp} proves that if some $p_i$ is 2, then $\calP$ is flat.
	\end{proof}

	Next, we can rule out infinite polytopes.
	
	\begin{proposition}
	\label{no-infinite-coxeter}
	If $\calP$ is a self-dual regular polytope such that $\G(\calP)$ is an infinite string
	Coxeter group, then $\calP$ is externally self-dual.
	\end{proposition}
			
	\begin{proof}
	Lemma 2.14 in \cite{cox-aut} proves that an infinite Coxeter group with no finite irreducible
	components has no nontrivial inner automorphisms that realize any graph automorphisms of the
	Coxeter diagram. Since self-duality is a graph automorphism, it follows that $\calP$
	cannot be internally self-dual.
	\end{proof}

	We can now easily cover the remaining self dual string Coxeter groups.
	
	\begin{theorem}
	\label{universal}
	The only internally self-dual regular polytopes such that $\G(\calP)$ is a string Coxeter group are simplices and $p$-gons
	with $p$ odd.
	\end{theorem}

	\begin{proof}
	In light of Corollary~\ref{no-degenerate} and Proposition~\ref{no-infinite-coxeter}, the only possibilities
	left to consider are simplices, polygons, and the 24-cell $\{3, 4, 3\}$. Propositions \ref{symmetric} and \ref{p-gons}
	 establish the claim for simplices and $p$-gons. Using a computer algebra system, we
	can verify that $\{3, 4, 3\}$ is not internally self-dual.
	\end{proof}

\subsection{Restrictions on the automorphism group}
\label{restrictions}

Based on the data in Table~\ref{MarstonData} and the restrictions from the previous section, it seems that regular internally self-dual polytopes could be relatively rare, especially in ranks other than three.  Before exploring some existence results about internally self-dual polytopes, let us discuss a few natural questions that arise while looking for examples.  First, one might want to know if the existence of internally or externally self-dual polytopes with a certain group $\Gamma$ as an automorphism group might depend on the abstract class of $\Gamma$.  In a simple way, the answer to this question is yes, in that if the automorphism group of $\Gamma$ is isomorphic to $\Gamma$ itself, then every group automorphism is inner, and there can be no external dualities of the polytope (as seen in Proposition~\ref{symmetric}).  Otherwise, it seems that the abstract structure of the group does not provide insight into whether self-dual regular polytopes will be either externally or internally self-dual.  In particular, there are both internally and externally self-dual regular polytopes with simple groups as their automorphism groups; examples of which can be easily found by considering alternating groups, for example.  

Second, while looking for internally self-dual polytopes of higher rank, a natural question is whether they are built from only internally self-dual polytopes of lower rank.  For example, must the medial sections of an internally self-dual polytope be internally self-dual themselves?  (The \emph{medial section} of a polytope is a section $F/v$ where $F$ is a facet and $v$ is a vertex.) Consider the unique self-dual regular polytope of type of rank four with an alternating group acting on nine points as its automorphism group (see Figure 4 of \cite{high-rank-alternating}).   This is easily seen to be internally self-dual, as the duality is realized as an even permutation on the nine points.  However, this polytope is of type $\{5,6,5\}$, and so its medial sections, being hexagons, are not internally self-dual; see Proposition~\ref{p-gons}.

Finally, Proposition~\ref{int-self-dual-covers} says that if a regular polytope is internally self-dual, then the only regular polytopes that it covers are also internally self-dual.  This is a stringent requirement, so one might hope that the converse would be true. However, there are externally self-dual polytopes that only cover internally-self dual polytopes.  For example, the unique polyhedron $\calP$ of type $\{5, 5\}$ with 320 flags is externally self-dual, and it double-covers the unique polyhedron $\calQ$ of type $\{5, 5\}$ with 160 flags, which is internally self-dual. Furthermore, every quotient of the former (polyhedral or not) filters through the latter, since the kernel of the covering from $\calP$ to $\calQ$ was the unique minimal normal subgroup of $\G(\calP)$. So $\calP$ only covers internally self-dual polyhedra, despite being externally self-dual itself.

\comment{
	\begin{proposition}
	Suppose $\calP = \{p_1, \ldots, p_{n-1}\}$. If $\calP$ is internally self-dual, then every
	$p_i$ is odd.
	\end{proposition}
	\begin{proof}
	Since $\calP$ is internally self-dual, its automorphism group $\G(\calP)$ has a dualizing automorphism $\alpha$, 
	and this automorphism remains dualizing in any quotient of $\G(\calP)$. Let $G = \langle \ol{\rho_0}, \ldots,
	\ol{\rho_{n-1}} \rangle$ be the abelianization of $\G(\calP)$, and let $\ol{\alpha}$ be the image of $\alpha$ in $G$. Then, for each $0 \leq i \leq n-1$, we have
			\[ \ol{\rho_i} \ol{\alpha} = \ol{\alpha} \ol{\rho_i} = \ol{\rho_{n-i-1}} \ol{\alpha}, \]
	where the first equality follows because $G$ is abelian, and the second because $\ol{\alpha}$ is dualizing.
	So $\ol{\rho_i} = \ol{\rho_{n-i-1}}$ in $G$. Now, $G$ is the degenerate string Coxeter
	group $[s_1, \ldots, s_{n-1}]$, where $s_i$ is 1 if $p_i$ is odd, and 2 if $p_i$ is even. 
	The only way for $\ol{\rho_0}$ to equal $\ol{\rho_{n-1}}$ is if every $s_i$ is 1 and all of the generators are identified. So every $p_i$ must be odd.
	\end{proof}
}
		
\section{Examples of internally self-dual polytopes \label{Examples}}		
In this section we will prove the existence of internally and externally self-dual polytopes with various characteristics. 
We mainly focus on rank 3, but higher rank polytopes are also constructed.

\subsection{Rank three}

	First we construct a few families of internally self-dual regular polyhedra. Our main result is the following.

	\begin{theorem} \label{main-thm}
	For each $p \geq 3$ (including $p = \infty$), there is an internally self-dual regular polyhedron of type $\{p, p\}$,
	and for each $p \geq 4$ (including $p = \infty$), there is an externally self-dual regular polyhedron of type $\{p, p\}$.
	Furthermore, if $p$ is even, then there are infinitely many internally and externally self-dual regular polyhedra of type $\{p, p\}.$
	\end{theorem}
	
	\comment{
		Before we begin the proof of this theorem, we note that there are also externally self dual regular polyhedra of type $\{p,p\}$ for $p$ odd.  These can be constructed by finding self-dual polyhedra with automorphism groups that are subgroups of alternating groups, but with dualities that are odd permutations.  However, the proofs of the structures of these examples would add length and complications to this paper that we chose to avoid.
		
		{\bf Note: I am not too happy with this exact wording, but I think it is better than nothing.}
	}

	We will focus on constructing internally self-dual regular polyhedra; then Theorem~\ref{universal}
	and Corollary~\ref{even-type} will take care of the rest. First we will show that there is an
	internally self-dual polyhedron of type $\{p, p\}$ for each $p \geq 3$. The data from \cite{conder-atlas}
	provides examples for $3 \leq p \leq 12$.
	We will construct a family that covers $p \geq 7$. We start with a simple lemma.

	\begin{lemma}
	\label{distinct-cycles}
	Suppose $\pi_1$ and $\pi_2$ are distinct permutations that act cyclically on $n$ points and that
	$\pi_1^{d_1} = \pi_2^{d_2}$ for some $d_1$ and $d_2$. Suppose that for some positive integer
	$k$, there is a unique point $i$ such that $i \pi_1^j = i \pi_2^j$ for all $j$ with $1 \leq j \leq k$.
	Then $d_1 = d_2 = 0 \ (\textrm{mod }n)$.
	\end{lemma}
	
	\begin{proof}
	Since $\pi_1^{d_1} = \pi_2^{d_2}$, it follows that $\pi_2$ and $\pi_1^{d_1}$ commute. Then, for $1 \leq j \leq k$,
	\begin{align*}
	(i \pi_1^{d_1}) \pi_2^j &= (i \pi_2^j) \pi_1^{d_1} \\
	&= (i \pi_1^j) \pi_1^{d_1} \\
	&= (i \pi_1^{d_1}) \pi_1^j.
	\end{align*}
	That is, $\pi_1^j$ and $\pi_2^j$ act the same way on $(i \pi_1^{d_1})$ for $1 \leq j \leq k$. 
	By assumption, $i$ was the only point such that $\pi_1^j$ and $\pi_2^j$ act the same way on that
	point for $1 \leq j \leq k$. It follows that $\pi_1^{d_1}$ (which is equal to
	$\pi_2^{d_2}$) fixes $i$, which implies that $d_1 = d_2 = 0 \  (\textrm{mod } n)$. 
	\end{proof}
	
	\begin{theorem}
	\label{all-p-example}
	For each $p \geq 7$, there is an internally self-dual polyhedron
	of type $\{p, p\}$ such that $(\rho_0 \rho_2 \rho_1)^6$ is dualizing.
	\end{theorem}
	
	\begin{proof}
	We will construct a permutation representation graph, and then show that it is a CPR graph.
	If $p$ is odd, then consider the following permutation representation graph:
     $$ \xymatrix{
	     *+[o][F]{4}     \ar@{=}[dd]_0^2   \ar@{-}[rr]^1	&	&	*+[o][F]{5}     \ar@{-}[dl]_2   \ar@{-}[dr]^0&		&	&	&	&	&		&								\\	
				&	*+[o][F]{1}   \ar@{-}[dr]_0&		&*+[o][F]{6}   \ar@{-}[dl]^2 \ar@{-}[r]^1	&*+[o][F]{7}	  \ar@{=}[r]^0_2& *+[o][F]{8}	 \ar@{-}[r]^1 &*+[o][F]{9}	 \ar@{.}[r]&	*+[o][F]{_{p-2}}  \ar@{=}[r]^0_2 & *+[o][F]{_{p-1}}  \ar@{-}[r]^1 & *+[o][F]{_p}										\\
		*+[o][F]{3}    \ar@{-}[rr]_1 			&	&*+[o][F]{2}    		&	&	&		&	&	&	&									}$$
	If $p$ is even, then then instead consider the following.
     $$ \xymatrix{
	     *+[o][F]{4}     \ar@{=}[dd]_0^2   \ar@{-}[rr]^1	&	&	*+[o][F]{5}     \ar@{-}[dl]_2   \ar@{-}[dr]^0&		&	&	&	&	&		&								\\	
				&	*+[o][F]{1}   \ar@{-}[dr]_0&		&*+[o][F]{6}   \ar@{-}[dl]^2 \ar@{-}[r]^1	&*+[o][F]{7}	  \ar@{=}[r]^0_2& *+[o][F]{8}	 \ar@{-}[r]^1 &*+[o][F]{9}	 \ar@{.}[r]&	*+[o][F]{_{p-2}}  \ar@{-}[r]^1 & *+[o][F]{_{p-1}}  \ar@{=}[r]^0_2 & *+[o][F]{_p}										\\
		*+[o][F]{3}    \ar@{-}[rr]_1 			&	&*+[o][F]{2}    		&	&	&		&	&	&	&									}$$	
	In both cases, it is easy to see that the group is a string group generated by involutions.  To verify this, we only need to notice that the subgraph induced by edges of labels 0 and 2, consists of connected components that are either isolated vertices, double edges, or squares with alternating labels.
	
	If $p$ is odd, then
	\[ \rho_0 \rho_2 \rho_1 = (1,7,6)(2,4,5,3)(8,9)(10,11) \cdots (p-1, p), \]
	and if $p$ is even, then
	\[ \rho_0 \rho_2 \rho_1 = (1,7,6)(2,4,5,3)(8,9)(10,11) \cdots (p-2, p-1). \]
	In both cases, it follows that $(\rho_0 \rho_2 \rho_1)^6 = (2,5)(3,4)$. It is simple to show that this is
	in fact a dualizing automorphism. In other words, $(\rho_0 \rho_2 \rho_1)^6 \rho_i$ acts the same on every vertex
	as $\rho_{2-i} (\rho_0 \rho_2 \rho_1)^6$.
	
	It remains to show that $\langle \rho_0, \rho_1, \rho_2 \rangle$ is a string C-group.  Following \cite[Proposition 2E16]{arp}, it suffices to show that
	$\langle \rho_0, \rho_1 \rangle \cap \langle \rho_1, \rho_2 \rangle = \langle \rho_1 \rangle$. 

	Let $\varphi$ be in the intersection, $\varphi \notin \langle \rho_1 \rangle$. Without loss
	of generality, we may assume that $\varphi = (\rho_0 \rho_1)^{d_1}$ for some $d_1$ (since if $\varphi$
	is odd, then $\varphi \rho_1$ is also in the intersection and can be written like that). 
	Now, $\varphi$ is also in $\langle \rho_1, \rho_2 \rangle$. If $\varphi = \rho_1 (\rho_2 \rho_1)^{d_2}$
	for some $d_2$, then $\varphi^2 = \eps$, and thus $(\rho_0 \rho_1)^{2d_1} = \eps$. 
	If $p$ is odd, then this only happens if $\varphi$ is itself the identity, contrary to our assumption.
	So in that case we may assume that $\varphi = (\rho_2 \rho_1)^{d_2}$ for some $d_2$.
	If $p$ is even, then in principle, it is possible that $\varphi = (\rho_0 \rho_1)^{p/2}$,
	and so it could happen that $\varphi = \rho_1 (\rho_2 \rho_1)^{d_2}$.

	Suppose $p$ is odd, $p \geq 9$. Then
	\[ \rho_0 \rho_1 = (1, 3, 5, 7, 9, \ldots, p, p-1, p-3, \ldots, 6, 4, 2) \]
	and
	\[ \rho_2 \rho_1 = (1, 4, 2, 7, 9, \ldots, p, p-1, p-3, \ldots, 6, 3, 5). \]
	We note that these cycles act the same way on 3, on 4, and on 7 through $p$.
	Indeed, 7 is the start of a unique longest sequence of points on which $\rho_0 \rho_1$
	and $\rho_2 \rho_1$ act, and so we can apply Lemma~\ref{distinct-cycles} with $i = 7$ and $k = p-6$.
	It follows that there is no nontrivial equation of the form $(\rho_0 \rho_1)^{d_1} = (\rho_2 \rho_1)^{d_2}$.
	
	Now, suppose $p$ is even. Then
	\[ \rho_0 \rho_1 = (1, 3, 5, 7, 9, \ldots, p-1, p, p-2, p-4, \ldots, 6, 4, 2) \]
	and
	\[ \rho_2 \rho_1 = (1, 4, 2, 7, 9, \ldots, p-1, p, p-2, p-4, \ldots, 6, 3, 5). \]
	As in the odd case, $(\rho_0 \rho_1)^{d_1}$ cannot
	equal $(\rho_2 \rho_1)^{d_2}$, by Lemma~\ref{distinct-cycles} with $i = 7$ and $k = p-6$.
	We still need to rule
	out the case $(\rho_0 \rho_1)^{p/2} = \rho_1 (\rho_2 \rho_1)^{d_2}$. Note that
	$(\rho_0 \rho_1)^{p/2}$ always sends 1 to $p$. In order for $\rho_1 (\rho_2 \rho_1)^{d_2}$
	to do the same thing, we would need $d_2 = p/2$ as well. But then $(\rho_0 \rho_1)^{p/2}$ sends 3
	to $p-2$, whereas $\rho_1 (\rho_2 \rho_1)^{p/2}$ sends 3 to $p-3$. So that rules
	out this case, proving that the intersection condition holds.
	
	The remaining cases where $p=7, 8$ can be verified using a computer algebra system.
	\end{proof}

	Using the polyhedra built in Theorem~\ref{all-p-example} as a base, we can construct an
	infinite polyhedron that is internally self-dual.
	
	\begin{theorem}
	\label{infinite-mix}
	Consider a family of internally self-dual regular polyhedra $\{\calP_i\}_{i=1}^{\infty}$.
	Let $\G(\calP_i) = \langle \rho_0^{(i)}, \rho_1^{(i)}, \rho_2^{(i)} \rangle$.
	Suppose that there is a finite sequence $j_1, \ldots, j_m$ such that
	$\rho_{j_1}^{(i)} \cdots \rho_{j_m}^{(i)}$ is dualizing in each $\G(\calP_i)$.
	Then $\calP = \calP_1 \mix \calP_2 \mix \cdots$ is an infinite internally self-dual
	regular polyhedron.
	\end{theorem}

	\begin{proof}
	It is clear that $\calP$ is infinite.
	Let $\rho_0 = (\rho_0^{(1)}, \rho_0^{(2)}, \ldots)$, 
	$\rho_1 = (\rho_1^{(1)}, \rho_1^{(2)}, \ldots)$, 
	$\rho_2 = (\rho_2^{(1)}, \rho_2^{(2)}, \ldots)$.
	To show that $\calP$ is a polyhedron, we need to show that $\langle \rho_0, \rho_1, \rho_2 \rangle$
	satisfies the intersection condition. In particular, by Proposition~2E16 of \cite{arp}, the only intersection that is nontrivial
	to check is $\langle \rho_0, \rho_1 \rangle \cap \langle \rho_1, \rho_2 \rangle$.
	Suppose $\varphi \in \langle \rho_0, \rho_1 \rangle \cap \langle \rho_1, \rho_2 \rangle$.
	For each $i$, let $\varphi_i$ be the natural projection of $\varphi$ in $\G(\calP_i)$, where
	we send each $\rho_j$ to $\rho_j^{(i)}$. Since $\calP_i$ is a polyhedron, each $\varphi_i$
	is either $\eps$ or $\rho_1^{(i)}$. Now, since $\langle \rho_0, \rho_1 \rangle$ is dihedral,
	the automorphism $\varphi$ is either even or odd. Furthermore, its projection $\varphi_i$ into
	the dihedral group $\langle \rho_0^{(i)}, \rho_1^{(i)} \rangle$ must have the same parity as
	$\varphi$ itself. Therefore, every $\varphi_i$ must have the same parity, and so either
	$\varphi_i = \eps$ for every $i$, or $\varphi_i = \rho_1^{(i)}$ for every $i$. In the first case,
	$\varphi = \eps$, and in the second, $\varphi = \rho_1$. This proves the intersection condition.
	
	Finally, it is clear that $\rho_{j_1} \cdots \rho_{j_m}$ is dualizing in $\G(\calP)$,
	and so $\calP$ is internally self-dual.
	\end{proof}

	\begin{corollary}
	\label{infinite-example}
	There is an infinite internally self-dual regular polyhedron of type $\{\infty, \infty\}$, with
	dualizing automorphism $(\rho_0 \rho_2 \rho_1)^6$.
	\end{corollary}
	
	\begin{proof}
	Apply the construction of Theorem~\ref{infinite-mix} to the polyhedra in
	Theorem~\ref{all-p-example}.
	\end{proof}
	
	The infinite polyhedron of Corollary~\ref{infinite-example} is a little difficult to work with; we
	have neither a permutation representation nor a presentation for the automorphism group. With this example,
	however, we can now build a simpler example.
	
	\begin{corollary}
	Let $\G$ be the quotient of $[\infty, \infty]$ by the three relations 
	\[ (\rho_0 \rho_2 \rho_1)^6 \rho_i = \rho_{2-i} (\rho_0 \rho_2 \rho_1)^6, \textrm{ where $0 \leq i \leq 2$.} \]
	Then $\G$ is the automorphism group of an infinite internally self-dual regular polyhedron of type $\{\infty, \infty\}$.
	\end{corollary}
	
	\begin{proof}
	First, note that $\G$ covers the automorphism group of any polyhedron with dualizing automorphism
	$(\rho_0 \rho_2 \rho_1)^6$. Therefore, $\G$ covers the automorphism group of the polyhedron in
	Corollary~\ref{infinite-example}. Then the quotient criterion (see \cite[Thm. 2E17]{arp})
	shows that this is a polyhedron, since it covers the polyhedron
	in Corollary~\ref{infinite-example} without any collapse of the facet subgroup $\langle \rho_0, \rho_1 \rangle$.
	\end{proof}

	Now let us prove that there are infinitely many internally self-dual polyhedra of type $\{p, p\}$
	when $p \geq 4$ and $p$ is even. We first cover the case $p = 4$.
	
	\begin{proposition}
	\label{torus-map}
	The polyhedron $\{4, 4\}_{(s, 0)}$ is internally self-dual if and only if $s$ is odd.
	\end{proposition}
	
	\begin{proof}
	Let $\calP = \{4, 4\}_{(s, 0)}$, and let us identify the vertices of $\calP$ with
	$(\Z/s\Z)^2$. Let us choose the base flag to consist of the origin, the edge from the origin to $(1, 0)$,
	and the square $[0,1] \times [0,1]$. Then $\rho_0$ sends each vertex $(x, y)$ to $(1-x, y)$,
	$\rho_1$ sends $(x, y)$ to $(y, x)$, and $\rho_2$ sends $(x, y)$ to $(x, -y)$. 
	If $s$ is even, then $\rho_0$ does not fix any vertex, and so by Corollary~\ref{fixed-vertices}, 
	$\calP$ is externally self-dual. If $s$ is odd, say $s = 2k-1$, then the unique vertex fixed by $\langle 
	\rho_0, \rho_1 \rangle$ is $(k, k)$. Continuing with Algorithm~\ref{dual-flag-proc}, we want an
	edge that contains $(k, k)$ and $(k, k) \rho_2 = (k, -k) = (k, k+1)$. Finally, we want a square
	that contains that edge and its images under $\langle \rho_1, \rho_2 \rangle$, which consists of the 4
	edges bounding the square with corners $(k, k)$ and $(k+1, k+1)$. Thus, the flag that
	is dual to the base flag consists of the vertex $(k, k)$, the edge to $(k, k+1)$, and
	the square that also includes $(k+1, k+1)$. See Figure~\ref{4450}.
	\end{proof}
	
	\begin{figure}[h] 
$$\includegraphics[scale=.5]{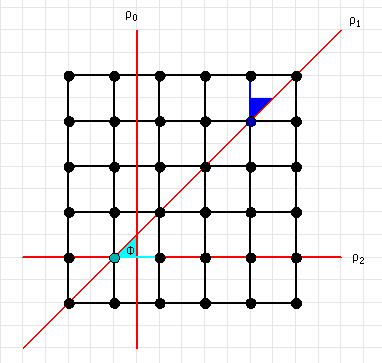}$$
\caption{Base and dual flags in the regular polyhedron $\{4,4\}_{(5,0)}$ \label{4450}}
\end{figure}

	We now construct a family of examples to cover the remaining cases.
	
	\begin{theorem}
	\label{infinite-even}
	For each even $p \geq 6$, there are infinitely many internally self-dual polyhedra of type $\{p, p\}$.
	\end{theorem}
	
	\begin{proof}
	For each even $p \geq 6$, we will construct a family of permutation representation graphs, and show that each one is the CPR graph
	of a distinct internally self-dual polyhedron of type $\{p, p\}$. First, consider the following permutation representation graph $G$:
	\[ \xymatrix {
	&	&	&	&	&	*+[o][F]{_{p-3}} \ar@{-}[dl]_2 \ar@{-}[dr]^0  & & & & & \\
	*+[o][F]{1} \ar@{-}[r]^{1} & *+[o][F]{2} \ar@{-}[r]^{0} & *+[o][F]{3} \ar@{-}[r]^{1} & *+[o][F]{4} \ar@{.}[r] & 
		*+[o][F]{_{p-4}} \ar@{-}[dr]_{0} & & *+[o][F]{_{p-1}} \ar@{-}[r]^{1} & *+[o][F]{_p} \ar@{.}[r] & *+[o][F]{_{2p-8}} \ar@{-}[r]^{2} & *+[o][F]{_{2p-7}} \ar@{-}[r]^{1} & *+[o][F]{_{2p-6}} \\
	&	&	&	&	&	*+[o][F]{_{p-2}} \ar@{-}[ur]_2 & & & & & \\
	}. \]
	Let $\rho_0, \rho_1, \rho_2$ be the permutations induced by edges of the appropriate label. 
	We have
	\begin{align*}
	\rho_0 &= (2,3)(4,5) \cdots (p-6, p-5) (p-4, p-2) (p-3, p-1), \\
	\rho_1 &= (1,2)(3,4) \cdots (p-5, p-4) (p-1, p) (p+1, p+2) \cdots (2p-7, 2p-6), \\
	\rho_2 &= (p-4, p-3) (p-2, p-1) (p, p+1)(p+2, p+3) \cdots (2p-8, 2p-7).
	\end{align*}
	It is again clear that $\langle \rho_0, \rho_1, \rho_2 \rangle$ is a string group generated by involutions. Next, note that $\rho_0 \rho_2 \rho_1$ interchanges $p-2$ with $p-3$, while cyclically permuting the remaining
	$2p-8$ points. Now consider $\sigma = (\rho_0 \rho_2 \rho_1)^{p-4}$. This interchanges $1$ with $2p-6$
	and thus $1 (\rho_0 \rho_2 \rho_1)^j$ with $(2p-6) (\rho_0 \rho_2 \rho_1)^j$, for each $j$. 
	Since the action of $(\rho_0 \rho_2 \rho_1)$ on vertices on the left is the mirror of the action on the right,
	it follows that $\sigma$ interchanges $i$ with $2p-5-i$ for $1 \leq i \leq p-4$, while fixing $p-2$ and $p-3$.
	
	\comment{
		Consider the permutation $\rho_0 \rho_2 \rho_1$. First of all, this permutation interchanges $p-3$ and
		$p-2$. It sends $p-4$ to $p$,
		and it sends $p-1$ to $p-5$. It sends 1 to 2 and $2p-6$ to $2p-7$. It sends every other even
		vertex $n$ to $n+2$ and every other odd vertex $n$ to $n-2$.
		
		Putting this all together, we note that repeated application of $\rho_0 \rho_2 \rho_1$ takes 1 to 2, then through
		all of the even vertices except for $p-2$, then from $2p-6$ to $2p-7$, then backwards
		through all of the odd vertices except for $p-3$, before coming back to 1.
		So $\rho_0 \rho_2 \rho_1$ is the product of a $(2p-8)$-cycle and a transposition.
	}
	
	We will now build a larger graph $\mathcal{X}$ using the one above as a building block. Let $k$ be an odd positive
	integer, and take $N := k(p-4)$ copies of the above graph, labeled $G_1, G_2, \ldots, G_{N}$,
	and arrange them cyclically. Let us use $(i, j)$ to mean vertex $i$ in $G_j$ (and where
	$j$ is considered modulo $N$ if necessary).
	We connect the graphs $G_j$ by adding edges labeled 1 from $(p-2, j)$ to $(p-3, j+1)$
	By Theorem 4.5 in~\cite{DanielCPR}, this is the CPR graph of a polyhedron,. Furthermore, if we erase the edges labeled 2, then the connected
	components either have 2 vertices or $p$ vertices. The latter consists of 
	the first $p-4$ vertices, then the bottom of a diamond, then the top of the next
	diamond, the right of that diamond, and one more vertex. The same happens if we erase the edges
	labeled 0, and so we get a polyhedron of type $\{p, p\}$.

	Let us now redefine $\rho_0, \rho_1,$ and $\rho_2$ as the permutations induced by edges of $\mathcal{X}$, and let
	$\sigma = (\rho_0 \rho_2 \rho_1)^{p-4}$ as before. The new $\sigma$ acts in exactly the same
	way on every vertex in every copy of the original CPR graph except for the top and bottom
	of the diamonds. Indeed, $\sigma$ takes $(p-3, j)$ to $(p-3, j-p+4)$ and it takes
	$(p-2, j)$ to $(p-2, j+p-4)$. Then the order of $\sigma$ is $2k$, since $\sigma$ to any
	odd power interchanges every $(1, j)$ with $(2p-6, j)$, and 
	$\sigma^k$ is the smallest power of $\sigma$ that fixes every $(p-3, j)$ and $(p-2, j)$.
	
	We claim that $\sigma^k$ is dualizing. To prove that, we need to show that $\rho_i \sigma^k = \sigma^k \rho_{2-i}$
	for $i = 0, 1, 2$. That is clearly true for all vertices other than the tops and bottoms
	of diamonds, because $\sigma^k$ acts as a reflection through the middle of the diagram, and
	this reflection also dualizes every label. Checking that $\rho_i \sigma_k$ and $\sigma^k \rho_{2-i}$
	act the same on the top and bottom of every diamond is then easy.
	
	Finally, we claim that the constructed polyhedra are distinct for each $k$. For this, it suffices
	to show that in each polyhedron, $k$ is the smallest positive integer such that $\sigma^k$ is dualizing.
	(In principle, if $\sigma^k$ is dualizing, it might also be true that $\sigma^m$ is dualizing for
	some $m$ dividing $k$.)
	In order for a power of $\sigma$ to dualize most vertices, it must be odd, since $\sigma^2$ fixes
	every vertex other than the tops and bottoms of diamonds. So consider $\sigma^m$ for some odd $m < k$.
	The permutation $\sigma^m \rho_1$ sends $(p-3, 1)$ to $(p-2, m(p-4))$,
	whereas $\rho_1 \sigma^m$ sends $(p-3, 1)$ to $(p-2, N-m(p-4)) = (p-2, (k-m)(p-4))$. 
	In order for these two points to be the same, we need $k-m = m$, so that $k = 2m$.
	But $k$ is odd, so this is impossible. So we get infinitely many internally self-dual polyhedra
	of type $\{p, p\}$.
	\end{proof}

	We can now prove our main result.
	
	\begin{proof}[Proof of Theorem~\ref{main-thm}]
	The data from \cite{conder-atlas}, combined with Theorem~\ref{all-p-example} and Corollary~\ref{infinite-example}, show that there are internally self-dual polyhedra of type $\{p, p\}$ for all $3 \leq p \leq \infty$. Theorem~\ref{universal}
	then shows that there are externally self-dual polyhedra of type $\{p, p\}$ for $4 \leq p \leq \infty$. 
	Proposition~\ref{torus-map} and Theorem~\ref{infinite-even} show that there are infinitely many internally
	self-dual polyhedra of type $\{p, p\}$ for $p$ even, and combining with Corollary~\ref{even-type}, we get
	infinitely many externally self-dual polyhedra of type $\{p, p\}$ for $p$ even as well.
	\end{proof}
	
\subsection{Higher ranks}

	Now that the rank three case is well established, let us consider internally self-dual regular polytopes of higher ranks.
	We will start by showing that there are infinitely many internally self-dual polytopes in every
	rank. By Theorem~\ref{universal}, we already know that the $n$-simplex is internally self-dual.
	It is instructive to actually show this constructively, using Algorithm~\ref{dual-flag-proc}.

	Consider the representation of the regular $n$-simplex $\calP$ as the convex hull of the
	points $e_1, e_2, \ldots, e_{n+1}$ in ${\mathbb R}^{n+1}$. Each $i$-face of $\calP$ is the convex
	hull of $i+1$ of the vertices, and each flag of $\calP$ can be associated to an
	ordering of the vertices $(e_{i_1}, e_{i_2}, \ldots, e_{i_{n+1}})$, where the
	$i$-face of the flag is the convex hull of the first $i+1$ vertices.
	
	Let us set the base flag $\Phi$ to be the flag corresponding to the ordering 
	$(e_1, e_2, \ldots, e_{n+1})$. Then each automorphism $\rho_i$ acts by switching
	coordinates $i+1$ and $i+2$, corresponding to a reflection in the plane $x_{i+1} = x_{i+2}$.
	In order to find a flag dual to $\Phi$, we need to first find a vertex that is
	fixed by $\langle \rho_0, \ldots, \rho_{n-2} \rangle$. The only such vertex is
	$e_{n+1}$. Next, we need an edge that is incident to $e_{n+1}$ and fixed
	by $\langle \rho_0, \ldots, \rho_{n-3} \rangle \times \langle \rho_{n-1} \rangle$.
	Since $\rho_{n-1}$ interchanges $e_n$ and $e_{n+1}$ and it must fix the edge,
	it follows that the edge must be incident on $e_n$. Continuing in this manner,
	it is easy to see that the dual flag must be $(e_{n+1}, e_n, \ldots, e_1)$.
	
	Now let us find the dualizing automorphism of $\G(\calP)$. We can identify $\G(\calP)$
	with the symmetric group on $n+1$ points, where $\rho_i$ is the transposition
	$(i+1, i+2)$. We noted above that the dual of the base flag simply reversed
	the order of the vertices. So the dualizing automorphism of $\G(\calP)$ can be
	written as
	\[ (\rho_0 \rho_1 \cdots \rho_{n-1}) (\rho_0 \rho_1 \cdots \rho_{n-2}) \cdots (\rho_0 \rho_1) (\rho_0), \]
	as this ``bubble-sorts'' the list $(1, 2, \ldots, n+1)$ into its reverse.

	Here is another example of high-rank internally self-dual polytopes. That they are internally self-dual
	follows from Proposition~\ref{symmetric}.

	\begin{proposition}
	For each $n \geq 5$ the regular $n$-polytope with group $S_{n+3}$ described in \cite[Proposition 4.10]{Extension_n-4}, is internally self dual.
	\end{proposition}

	$$ \xymatrix@-1pc{ *+[o][F]{}   \ar@{-}[r]^1  & *+[o][F]{}  \ar@{-}[r]^0    &  *+[o][F]{}   \ar@{-}[r]^1  & *+[o][F]{}  \ar@{-}[r]^2 & *+[o][F]{}  \ar@{-}[r]^3 & *+[o][F]{}  \ar@{-}[r]^4  & *+[o][F]{} \ar@{.}[r] & *+[o][F]{}  \ar@{-}[r]^{n-4} & *+[o][F]{}  \ar@{-}[r]^{n-3} & *+[o][F]{}  \ar@{-}[r]^{n-2} & *+[o][F]{}  \ar@{-}[r]^{n-1} & *+[o][F]{}  \ar@{-}[r]^{n-2} & *+[o][F]{}}$$

	The CPR graph of these polytopes is shown above. Each of this polytopes is obtained from a simplex, by first taking the Petrie contraction and then dualizing and taking another Petrie contraction.   (The Petrie contraction of a string group generated by involutions $\langle \rho_0, \rho_1, \rho_2, \ldots, \rho_{n-1} \rangle$ is the group generated by $\langle \rho_1, \rho_0 \rho_2, \rho_3, \ldots, \rho_{n-1} \rangle$.)

	\comment{
	
		\begin{lemma}
		Suppose $\calP$ is a regular self-dual polytope with base flag $\Phi$. Then $\Psi$ is dual to $\Phi$
		if and only if, for each $i$-face $G_i$ of $\Psi$, the subgroup $\langle \rho_j \mid j \neq {n-i-1} \rangle$
		fixes $G_i$.
		\end{lemma}
			
		\begin{proof}
		The necessity is clear. If $\langle \rho_j \mid j \neq n-i-1 \rangle$ fixes $G_i$, then
		$\rho_{n-i-1}$ cannot also fix $G_i$, because then $G_i$ would be fixed by every automorphism.
		So $\rho_{n-i-1}$ moves $G_i$ while fixing every other face of $\Psi$, which shows that
		$\Psi \rho_{n-i-1} = \Psi^i$.
		\end{proof}

		{\bf Am I using that lemma? Or is that just another way of putting the "building a dual flag" construction
		in the last section?}

		{\bf No good reason to put it in this section.  Right now the algorithm only is written to guarantee dual flags for vertex describable polytopes.  }
		
	}

	The cubic toroids (described in \cite[Section 6D]{arp}) provide an infinite family of internally self-dual polytopes
	with automorphism groups other than the symmetric group.
	
	\begin{theorem}
	The regular $n+1$-polytope $\{4, 3^{n-2}, 4\}_{(s, 0^{n-1})}$ is internally self-dual if and only if $s$ is odd.
	\end{theorem}
		
	\begin{proof}
	Let us take the vertex set of $\calP$ to be $(\Z/s\Z)^n$. For $0 \leq i \leq n$, let $G_i$ be the $i$-face of $\calP$ 
	containing vertices ${\bf 0}, e_1, e_1 + e_2, \ldots, e_1 + e_2 + \cdots + e_i$, where 
	$\{e_i\}_{i = 1}^{n}$ is the standard basis. Let $\Phi = (G_0, \ldots, G_{n})$ be our base flag.
	Then the generators of $\G(\calP)$ can be described geometrically as follows (taken from \cite[Eq. 6D2]{arp}).
	$\rho_0$ sends $(x_1, \ldots, x_n)$ to
	$(1-x_1, x_2, \ldots, x_n)$, $\rho_n$ sends $(x_1, \ldots, x_n)$ to $(x_1, \ldots, x_{n-1}, -x_n)$, and
	for $1 \leq i \leq n-1$, $\rho_i$ interchanges $x_i$ and $x_{i+1}$.
		
	We now try to build a flag that is dual to $\Phi$, using Algorithm~\ref{dual-flag-proc}.
	First, we need to find a vertex that is fixed by $\langle \rho_0, \ldots, \rho_{n-1} \rangle$.
	In order for $(x_1, \ldots, x_n)$ to be fixed by $\rho_0$, we need for $x_1 \equiv 1-x_1$ (mod $s$);
	in other words, $2x_1 \equiv 1$ (mod $s$). That has a solution if and only if $s$ is odd, so
	that already establishes that when $s$ is even, the polytope $\calP$ is not internally self-dual.
	On the other hand, if $s$ is odd, then we can take $x_1 = (s+1)/2$ as a solution. Now, in order
	for $\rho_1, \ldots, \rho_{n-1}$ to also fix this vertex, we need all of the coordinates to be the same.
	So we pick $F_0 = ((s+1)/2, \ldots, (s+1)/2)$.
		
	Next we need an edge incident on $F_0$ and $F_0 \rho_n$. The latter is simply $((s+1)/2, \ldots, (s+1)/2, (s-1)/2)$,
	which is indeed adjacent to $F_0$, and that gives us our edge $F_1$. To pick $F_2$, we look at the orbit
	$G_1 \langle \rho_{n-1}, \rho_n \rangle$; this gives us the square whose 4 vertices are $((s+1)/2, \ldots,
	(s+1)/2, \pm (s+1)/2, \pm (s+1)/2)$. In general, we take $F_i$ to be the $i$-face such that
	its vertices are obtained from $F_0$ by any combination of sign changes in the last $i$ coordinates.
	Then it is clear that $F_i$ is fixed by $\langle \rho_0, \ldots, \rho_{n-i-1}, \rho_{n-i+1}, \ldots, \rho_n \rangle$,
	and thus we have a dual flag to $\Phi$.
	\end{proof}

To show that there are examples of internally self-dual polytopes other than the toroids and simplices (and self-dual petrie contracted relatives) in each rank $r \geq 5$, we give the following family of examples and prove that they yield string C-groups which are internally self-dual.  (There are other examples in rank 4 as well, such as the polytope of type
$\{5, 6, 5\}$ that was mentioned at the end of Section~\ref{restrictions}.) 

\begin{lemma}
\label{n-3plus}
For each $n \geq 4$, the following permutation representation graph is the CPR graph of a regular $n$-polytope with automorphism group $\G$ isomorphic to the symmetric group on $n+3$ points.
\end{lemma}

\begin{center}
$ \xymatrix@-1pc{
     *+[o][F]{}   \ar@{-}[r]^2  \ar@{-}[d]_0     & *+[o][F]{}  \ar@{-}[r]^1 \ar@{-}[d]^0  & *+[o][F]{}  \ar@{-}[r]^2 & *+[o][F]{}  \ar@{-}[r]^3 & *+[o][F]{}  \ar@{-}[r]^4 & *+[o][F]{}  \ar@{..}[r] & *+[o][F]{}  \ar@{-}[r]^{n-1}& *+[o][F]{}  \\
     *+[o][F]{}  \ar@{-}[r]_2     & *+[o][F]{}    &  & & & & &  }$
  \end{center}
  
\begin{proof}  
For $n \geq 6$ this is shown in \cite[Proposition 4.8]{Extension_n-4}.  The remaining cases can either be checked by hand or using a computer algebra system.
\end{proof}

\begin{lemma}
\label{n-4plus}
For each $n \geq 6$, the following permutation representation graph is the CPR graph of a regular $n$-polytope with automorphism group $\G$ isomorphic to the symmetric group on $n+4$ points.  Furthermore, the element $(\rho_{n-2} \rho_{n-1} \rho_{n-2} \rho_{n-3} \rho_{n-2})$ is a five cycle with the last five points in its support.
\end{lemma}

\begin{center}
$ \xymatrix@-1pc{
     *+[o][F]{}   \ar@{-}[r]^2  \ar@{-}[d]_0     & *+[o][F]{}  \ar@{-}[r]^1 \ar@{-}[d]^0  & *+[o][F]{}  \ar@{-}[r]^2 & *+[o][F]{}  \ar@{-}[r]^3 & *+[o][F]{}  \ar@{-}[r]^4 & *+[o][F]{}  \ar@{..}[r] & *+[o][F]{}  \ar@{-}[r]^{n-3}   & *+[o][F]{}   \ar@{-}[r]^{n-2}   & *+[o][F]{}   \ar@{-}[r]^{n-1}   & *+[o][F]{}   \ar@{-}[r]^{n-2} & *+[o][F]{}\\
     *+[o][F]{}  \ar@{-}[r]_2     & *+[o][F]{}  &&  &  & & & & & &  }$
  \end{center}
  
\begin{proof}  
For $n \geq 7$ the fact that this is a CPR graph is shown in \cite[Proposition 4.14]{Extension_n-4}.  The remaining case can either be checked by hand or using a computer algebra system.  The structure of $(\rho_{n-2} \rho_{n-1} \rho_{n-2} \rho_{n-3} \rho_{n-2})$ can easily be checked by hand.
\end{proof}

	\begin{theorem}
For each $n \geq 5$, the following permutation representation graph is the CPR graph of an internally self-dual regular $n$-polytope with automorphism group $\G$ acting on $n+5$ points.	\end{theorem}
 
\begin{center}
$ \xymatrix@-1pc{
     *+[o][F]{3}   \ar@{-}[r]^2  \ar@{-}[d]_0     & *+[o][F]{4}  \ar@{-}[r]^1 \ar@{-}[d]^0  & *+[o][F]{5}  \ar@{-}[r]^2 & *+[o][F]{6}  \ar@{-}[r]^3 &  *+[o][F]{7}  \ar@{..}[r] &  *+[o][F]{_{n-1}} \ar@{-}[r]^{n-4}& *+[o][F]{n} \ar@{-}[r]^{n-3}  & *+[o][F]{_{n+1}} \ar@{-}[r]^{n-2} & *+[o][F]{_{n+2}} \ar@{-}[r]^{n-3}  \ar@{-}[d]_{n-1} & *+[o][F]{_{n+3}}  \ar@{-}[d]^{n-1} \\
     *+[o][F]{2}  \ar@{-}[r]_2     & *+[o][F]{1}     & & & & & & &  *+[o][F]{_{n+5}}  \ar@{-}[r]_{n-3}     & *+[o][F]{_{n+4}}  }$
  \end{center}

\begin{proof}

First we notice that these graphs have slightly different structure depending on whether $n$ is odd or even.  The two smallest cases seen below, of ranks 5 and 6, are verified to yield string C-groups using a computer algebra system.  In the rest of the proof we will assume that $n \geq 7$.

\begin{center}
\begin{tabular}{cc}
$ \xymatrix@-1pc{
     *+[o][F]{3}   \ar@{-}[r]^2  \ar@{-}[d]_0     & *+[o][F]{4}  \ar@{-}[r]^1 \ar@{-}[d]^0  & *+[o][F]{5}  \ar@{-}[r]^2 & *+[o][F]{6}  \ar@{-}[r]^3  & *+[o][F]{7} \ar@{-}[r]^{2}  \ar@{-}[d]_{4} & *+[o][F]{8}  \ar@{-}[d]^{4} \\
     *+[o][F]{2}  \ar@{-}[r]_2     & *+[o][F]{1}       & & &   *+[o][F]{10}  \ar@{-}[r]_{2}     & *+[o][F]{9}  }$ \hspace*{1cm}
&
 \hspace*{1cm} $ \xymatrix@-1pc{
     *+[o][F]{3}   \ar@{-}[r]^2  \ar@{-}[d]_0     & *+[o][F]{4}  \ar@{-}[r]^1 \ar@{-}[d]^0  &  *+[o][F]{5} \ar@{-}[r]^{2}& *+[o][F]{6} \ar@{-}[r]^{3}  & *+[o][F]{7} \ar@{-}[r]^{4} & *+[o][F]{8} \ar@{-}[r]^{3}  \ar@{-}[d]_{5} & *+[o][F]{9}  \ar@{-}[d]^{5} \\
     *+[o][F]{2}  \ar@{-}[r]_2     & *+[o][F]{1}     & & & &  *+[o][F]{11}  \ar@{-}[r]_{3}     & *+[o][F]{10}  }$
\end{tabular}
\end{center}

By design, $\G$ is self-dual, since interchanging every label $i$ with $n-1-i$ is a symmetry of the graph,
corresponding to a reflection through a line that goes through the middle edge (if $n$ is odd) or through the middle node (if $n$ is even).
Let us show that $\G$ is internally self-dual. 

First, when $n$ is odd and $n=2k+1$, the permutation $\pi_1 =  \rho_{k}$ interchanges the two nodes that are incident
on the middle edge, while fixing everything else. Then, setting $\pi_2 = (\rho_{k-1} \rho_{k+1}) \pi_1 
(\rho_{k+1} \rho_{k-1})$, we get that $\pi_2$ interchanges the next two nodes from the center, while
fixing everything else. Continuing this way, we can find permutations $\pi_1, \ldots, \pi_{k+3}$ that each
interchange a node with its dual while fixing everything else.  The product of all of these will be a dualizing automorphism in $\G$.

When $n$ is even and $n=2k$, the middle node is incident to edges labeled $k-1$ and $k$.
The permutation $\pi_1 = \rho_{k-1} \rho_{k} \rho_{k-1}$ is easily seen to interchange the two nodes that are incident
on the middle node, while fixing everything else. Then, setting $\pi_2 = (\rho_{k-2} \rho_{k+1}) \pi_1 
(\rho_{k+1} \rho_{k-2})$, we get that $\pi_2$ interchanges the two nodes at a distance of 2 from the middle node, while
fixing everything else. As in the odd case, we can continue to define permutations $\pi_1, \ldots, \pi_{k+2}$ that each
interchange a node with its dual, and the product of all of these will be a dualizing automorphism in $\G$.

Now we need to show that each $\G$ is a string C-group. It is clear that it is a string group generated by involutions, so we only need to show that it satisfies the intersection condition.  We will do this utilizing \cite[Proposition~2E16]{arp} by showing that its facet group and vertex figure group are string C-groups, and that their intersection is what is needed.

We will take advantage of the fact that $\G$ is self-dual by design, so the structure of each of its parabolic subgroups is the same as the ``dual" subgroup.  For any subset $S$ of $\{0, 1, \ldots, n-1\}$, we define $\G_S = \langle \rho_i
\mid i \not \in S \rangle$. The structure of $\G_S$ is determined by the structure of the given permutation representation graph after we delete all edges with labels in $S$. For example, the group $\G_{0,1,n-2,n-1}$ yields a permutation representation graph with a chain with labels $2$ through $n-3$, along with four isolated components; two labeled $2$ and two labeled $n-3$. It follows that this is the group obtained from a simplex by mixing with an edge, dualizing, and mixing with another edge.  It is a string C-group which is isomorphic to $2 \times 2 \times S_{n-3}$.  The groups $\G _{0,1,n-1}$ and $\G _{0,n-2,n-1}$ are maximal parabolic subgroups of the group obtained by mixing a string C-group from Lemma~\ref{n-3plus} with an edge.  Both of these groups are thus string C-groups, and can be shown to be isomorphic to $2 \times 2 \times S_{n-1}$.  Similarly, the groups $\G _{0,1,2}$ and $\G _{n-3,n-2,n-1}$ are both described by Lemma~\ref{n-3plus}.  They each are string C-groups isomorphic to $S_{n}$.   Each group $\G _{0,1}$ and $\G _{n-2,n-1}$ is the mix of a group from Lemma~\ref{n-3plus} with an edge, and can be shown to be isomorphic to $2 \times S_{n+1}$.  

Now, to show that $\G _{0,n-1}$ is a string C-group, we use the fact that both $\G _{0,1,n-1}$ and $\G _{0,n-2,n-1}$ are string C-groups acting on three orbits, and as a symmetric group on the largest orbit.  The intersection of these two groups cannot be any larger than the direct product of two groups of order two with the symmetric group acting on the points shared by the large orbits of both groups.  This is exactly the structure of $\G _{0,1,n-2, n-1}$, so $\G _{0,n-1}$ is a string C-group.  

Similarly, we show that $\G _0$ is a string C-group (the result then follows for $\G _{n-1}$, by self-duality). Since both $\G _{0,1}$ and $\G _{0,n-1}$ are string C-groups, we just need to show that their intersection is  $\G _{0,1, n-1}$.  This is true following the same logic as above, analyzing the orbits of each subgroup.

Also, following Lemma~\ref{n-4plus}, $\G _0$ is a symmetric group extended by a single transposition.  Furthermore, since element $(\rho_2 \rho_1  \rho_2 \rho_3 \rho_2)^5$ fixes all the nodes of the connected component of the graph of $\G_0$, and interchanges the nodes of the isolated edge labeled 2, it follows that this single transposition is in the group $\G_0$. Thus $\G _0 \cong 2 \times S_{n+3}$.

Finally, $\G$ is a string C-group since both $\G _0$ and $\G _{n-1}$ are string C-groups, and $\G _{0,1} \cong 2 \times 2 \times S_{n+1}$ is maximal in $\G _0 \cong 2 \times S_{n+3}$.  

\end{proof}

\section{Related problems and open questions \label{SelfPetrie}}
Some problems on the existence of internally self-dual polytopes remain open. Here are perhaps the most fundamental.

\begin{problem}
For each odd $p \geq 5$, describe an infinite family of internally self-dual regular polyhedra of type $\{p, p\}$, or prove that there are only finitely many internally self-dual regular polyhedra of that type.
\end{problem}

\begin{problem}
Determine the values of $p$ and $q$ such that there is an internally self-dual regular $4$-polytope of type $\{p, q, p\}$.
\end{problem}

\begin{problem}
Determine whether each self-dual $(n-2)$-polytope occurs as the medial section of an internally self-dual regular
$n$-polytope.
\end{problem}

To our knowledge, these problems are open even if we consider all self-dual regular polytopes,
rather than just the internally self-dual ones.

To what extent does the theory we have developed apply to transformations other than duality?
For example, the \emph{Petrie dual} of a polyhedron $\calP$, denoted $\calP^{\pi}$, is obtained from $\calP$ by
interchanging the roles of its facets and its Petrie polygons (see \cite[Sec. 7B]{arp}). If $\calP^{\pi} \cong \calP$,
then we say that $\calP$ is \emph{self-Petrie}. If $\calP$ is regular, with $\G(\calP) = \langle \rho_0, \rho_1, 
\rho_2 \rangle$, then $\calP$ is self-Petrie if and only if there is a group automorphism
of $\G(\calP)$ that sends $\rho_0$ to $\rho_0 \rho_2$, while fixing $\rho_1$ and $\rho_2$. We can then
say that $\calP$ is \emph{internally self-Petrie} if this automorphism is inner.

Working with internally self-Petrie polyhedra is not substantially different from working with
internally self-dual polyhedra, due to the following result.

\begin{proposition}
A regular polyhedron $\calP$ is internally self-Petrie if and only if $(\calP^*)^{\pi}$ is internally
self-dual.
\end{proposition}

\begin{proof}
Suppose that $\calP$ is (internally or externally) self-Petrie, and let $\G(\calP) = \langle \rho_0, \rho_1, \rho_2 \rangle$.
Then there is some $\pi \in \Aut(\G(\calP))$ such that $\rho_0 \pi = \rho_0 \rho_2$, $\rho_1 \pi =
\rho_1$, and $\rho_2 \pi = \rho_2$. Now, $\G(\calP^*) = \langle \rho_2, \rho_1, \rho_0 \rangle$, and then
$\G((\calP^*)^{\pi}) = \langle \rho_0 \rho_2, \rho_1, \rho_0 \rangle =: \langle \lambda_0,
\lambda_1, \lambda_2 \rangle$. Then it is easy to show that $\lambda_i \pi = \lambda_{2-i}$ for $i = 0, 1, 2$.
It follows that $(\calP^*)^{\pi}$ is self-dual. Furthermore, since $\G(\calP) \cong \G((\calP^*)^{\pi})$
as abstract groups and $\pi$ induces the self-Petriality of the former and the self-duality of the latter,
it follows that $\pi$ is either inner for both groups or outer for both. The result then follows.
\end{proof}

Perhaps there are other transformations of polytopes where it would make sense to discuss internal versus
external invariance under that transformation.
\section{Acknowledgements}
The computations done in this paper were verified independently using GAP \cite{gap} and
Magma \cite{Magma}.

\providecommand{\bysame}{\leavevmode\hbox to3em{\hrulefill}\thinspace}
\providecommand{\MR}{\relax\ifhmode\unskip\space\fi MR }
\providecommand{\MRhref}[2]{%
  \href{http://www.ams.org/mathscinet-getitem?mr=#1}{#2}
}
\providecommand{\href}[2]{#2}

\end{document}